\documentclass[12pt,american]{article}
\usepackage[T1]{fontenc}
\usepackage[utf8]{inputenc}
\usepackage{lmodern}
\usepackage{babel}
\usepackage{mathtools}
\usepackage{amsmath}
\usepackage{amsthm}
\usepackage{amssymb}
\usepackage{geometry}
\usepackage{setspace}
\usepackage{microtype}
\usepackage[hidelinks]{hyperref}
\allowdisplaybreaks
\numberwithin{equation}{section}

\newtheorem{theorem}{Theorem}[section]
\newtheorem{proposition}[theorem]{Proposition}
\newtheorem{lemma}[theorem]{Lemma}
\newtheorem{corollary}[theorem]{Corollary}
\theoremstyle{definition}
\newtheorem{definition}[theorem]{Definition}
\theoremstyle{remark}
\newtheorem{remark}[theorem]{Remark}

\title{Approximation by mixtures of multivariate Erlang distributions}
\author{Hien Duy Nguyen\\[0.4em]
\small School of Computing, Engineering and Mathematical Sciences, La Trobe University,\\
\small Bundoora, VIC 3086, Australia\\[0.2em]
\small Institute of Mathematics for Industry, Kyushu University,\\
\small Fukuoka 819-0395, Japan}
\date{}

\begin{document}
\maketitle 
\begin{abstract}
We prove that finite multivariate Erlang mixture densities with a
common rate parameter are dense in the class of probability densities
on $\mathbb{R}_{+}^{d}$ that belong to $L^{p}$, for every dimension
$d\in\mathbb{N}$ and every $1\le p<\infty$. The argument is constructive:
the one-dimensional Szász--Mirakjan--Kantorovich operator 
yields Erlang mixture approximations, and its tensor product yields multivariate
approximants with a common scale. We then obtain several quantitative
consequences. These include compact-set uniform approximation bounds
and, under local Hölder conditions of order $\alpha\in(0,1]$, rates
of order $n^{-\alpha/2}$ as the common scale $1/n$ tends to zero,
whole-domain convergence in weighted sup norms, weighted and unweighted
$L^{p}$ rates, and explicit rates for finite mixtures indexed by
the number of mixture components. In particular, if the approximating
density is required to have at most $K$ mixture components, then
on fixed compact cubes we obtain an algebraic rate of order $K^{-\alpha/(2d)}$;
in global weighted sup norms we obtain the explicit algebraic component-count
rate $K^{-\alpha/[2d(2d+\alpha)]}$; and for $1<p<\infty$ we obtain
corresponding weighted $L^{p}$ component-count rates. The results
strengthen the weak-approximation theory for multivariate Erlang mixture
distributions and yield immediate corollaries for broader classes
such as product-gamma mixtures. 
\end{abstract}
\noindent\textbf{Keywords:} multivariate Erlang mixtures; Erlang distributions; Szász--Mirakjan--Kantorovich
operator; density approximation; weighted $L^{p}$ approximation;
approximation rates.

\section{Introduction}

Modeling multivariate data on the positive orthant is a recurring
task in actuarial science, risk theory, reliability, and applied statistics.
Typical examples include claim severities, deductibles, and aggregate-loss
calculations in non-life insurance \cite{Boland2007,KaasGoovaertsDhaeneDenuit2008,KlugmanPanjerWillmot2012};
ruin-probability and retention calculations in classical risk theory
\cite{AsmussenAlbrecher2010,Boland2007,KaasGoovaertsDhaeneDenuit2008}; survival-time
and life-contingency calculations \cite{Promislow2015}; and dependent
loss components in multivariate insurance portfolios \cite{LeeLin2012,WillmotLin2011}.
In this setting, mixtures of Erlang distributions are attractive because
they are flexible, analytically tractable, and stable under several
operations that are important in insurance and risk modelling.

The actuarial literature contains a substantial body of work on this
class. In particular, \cite{LeeLin2012} introduced multivariate Erlang mixture
distributions and proved that they are dense, in the weak topology,
in the class of positive continuous distributions on $\mathbb{R}_{+}^{d}$.
Subsequent papers have studied structural properties, estimation,
and statistical inference for Erlang mixture
models \cite{GuiHuangLin2021,VerbelenAntonioClaeskens2016,WillmotLin2011,WillmotWoo2015,YinEtAl2019}.
In the univariate setting, the weak denseness of Erlang mixtures is
well known in the risk and actuarial literature and appears, for instance,
in \cite{AsmussenAlbrecher2010,LeeLin2010,Tijms2003}.

From the point of view of approximation theory, the same approximation
mechanism is classical. The relevant operator is the Szász--Mirakjan--Kantorovich
operator, whose origins go back to the works of \cite{Butzer1954,Mirakyan1941,Szasz1950};
see also the monograph \cite{AltomareCampiti1994}. In one dimension,
considerably stronger statements than weak convergence are available.
Uniform approximation on the half-line already appears in \cite{Szasz1950},
and $L^{p}$ convergence and rate results were obtained in \cite{Totik1983,Totik1985};
see also \cite{AltomareCappellettiLeonessa2013}. By contrast, although
multivariate Szász--Mirakjan-type operators do appear in the approximation-theory
literature \cite{BarbosuEtAl2010,MahmudovKara2024,YadavMishraMeherMursaleen2022},
we do not know of any work that explicitly transfers the one-dimensional
$L^{p}$ theory to the class of multivariate Erlang mixture densities
in arbitrary dimension.

Stronger convergence modes matter because weak convergence is often too weak for the questions
that arise in modelling and inference. By definition, weak convergence
only controls integrals against bounded continuous test functions
\cite{vanDerVaart1998}, and therefore does not control the density
itself. Even for absolutely continuous laws one may have weak convergence
while the densities remain separated in total variation. For a classical example,
on $[0,1]$ the densities $f_{n}(x)=1+\frac{1}{2}\sin(2\pi nx)$ converge
weakly to the uniform law, yet $\|f_{n}-1\|_{L^{1}([0,1])}=1/\pi$
for every $n$.

These failures are not merely cosmetic. In actuarial and risk modelling
one often needs uniformly accurate approximations of lower-orthant
probabilities for dependent portfolios \cite{LeeLin2012,WillmotLin2011},
tail or survival probabilities around deductibles and retentions \cite{AsmussenAlbrecher2010,Boland2007,KaasGoovaertsDhaeneDenuit2008},
limited expected values \cite{Boland2007,KlugmanPanjerWillmot2012}, and
stop-loss or treaty payoffs \cite{Boland2007,KaasGoovaertsDhaeneDenuit2008,KlugmanPanjerWillmot2012}.
The stronger modes of convergence obtained in this text provide control of such quantities. If $X_{n}$ and $X$ are $\mathbb{R}_{+}^{d}$-valued
random vectors with densities $f_{n}$ and $f$, and if $f_{n}\to f$
in $L^{1}(\mathbb{R}_{+}^{d})$, then for every measurable set $A\subseteq\mathbb{R}_{+}^{d}$,
\[
\left|\mathbb{P}(X_{n}\in A)-\mathbb{P}(X\in A)\right|\le\|f_{n}-f\|_{L^{1}(\mathbb{R}_{+}^{d})}.
\]
In the univariate case this yields 
\[
\sup_{x\ge0}|F_{n}(x)-F(x)|=\sup_{x\ge0}|\overline{F}_{n}(x)-\overline{F}(x)|\le\|f_{n}-f\|_{L^{1}(\mathbb{R}_{+})},
\]
and therefore, for every fixed $M>0$, 
\[
\sup_{0\le u\le M}\left|\mathrm{LEV}_{n}(u)-\mathrm{LEV}(u)\right|\le M\|f_{n}-f\|_{L^{1}(\mathbb{R}_{+})},
\]
where $\mathrm{LEV}_{n}(u)=\mathbb{E}[\min(X_{n},u)]$ and $\mathrm{LEV}(u)=\mathbb{E}[\min(X,u)]$
denote the limited expected value functions.

For $1<p<\infty$, with conjugate exponent $q=p/(p-1)$, Hölder's
inequality gives $L^{p}$ estimate 
\[
\left|\mathbb{P}(X_{n}\in A)-\mathbb{P}(X\in A)\right|\le\lambda_{d}(A)^{1/q}\|f_{n}-f\|_{L^{p}(\mathbb{R}_{+}^{d})}
\]
for every measurable set $A\subseteq\mathbb{R}_{+}^{d}$ with finite
Lebesgue measure $\lambda_{d}(A)$, and more generally 
\[
\left|\mathbb{E}[\psi(X_{n})]-\mathbb{E}[\psi(X)]\right|\le\|\psi\|_{L^{q}(\mathbb{R}_{+}^{d})}\|f_{n}-f\|_{L^{p}(\mathbb{R}_{+}^{d})}
\]
for every measurable payoff $\psi\in L^{q}(\mathbb{R}_{+}^{d})$.
For example, taking 
\[
\psi_{u,c,M}(x)=\left((x_{1}+\cdots+x_{d}-u)_{+}\wedge c\right)\mathbf{1}_{[0,M]^{d}}(x),\qquad a_{+}=\max\{a,0\},
\]
gives the bound for capped aggregate stop-loss payoffs on a fixed
compact window \cite{Boland2007,KaasGoovaertsDhaeneDenuit2008,KlugmanPanjerWillmot2012,LeeLin2012}.

Local $L^{\infty}$ control is useful in a different way. If $M>0$
and $\|f_{n}-f\|_{L^{\infty}([0,M]^{d})}$ is small, then every valuation
functional of the form 
\[
\Pi_{\psi}(g)=\int_{[0,M]^{d}}\psi(x)g(x)\,\mathrm{d}x
\]
with $\psi\in L^{1}([0,M]^{d})$ satisfies 
\[
|\Pi_{\psi}(f_{n})-\Pi_{\psi}(f)|\le\|\psi\|_{L^{1}([0,M]^{d})}\,\|f_{n}-f\|_{L^{\infty}([0,M]^{d})}.
\]
Thus compact-uniform density approximation controls whole families
of locally supported contract valuations. In the univariate survival
setting, if in addition $f_{n}\to f$ in $L^{1}(\mathbb{R}_{+})$
and locally uniformly on $[0,M]$, and if the survival function $\overline{F}(x)=\int_{x}^{\infty}f(t)\,\mathrm{d}t$
is bounded away from zero on $[0,M]$, then the hazard rates $h_{n}=f_{n}/\overline{F}_{n}$,
which are central objects in survival and life-contingency modelling
\cite{Promislow2015}, converge uniformly to $h=f/\overline{F}$ on
$[0,M]$. Accordingly, a constructive theory in $L^{1}$, $L^{p}$,
and local or weighted uniform norms is substantively stronger than
weak denseness. Since $L^{1}$ control yields convergence
in total variation, convergence in Hellinger distance follows.

The main purpose of this manuscript is therefore to record a direct
and self-contained transfer of the relevant approximation-operator
results to the Erlang mixture setting. The first theorem proves that
finite multivariate Erlang mixture densities with a common scale parameter
are dense in the class of probability densities on $\mathbb{R}_{+}^{d}$
that belong to $L^{p}$, for every $1\le p<\infty$. The proof is
constructive and rests on three simple ingredients: the classical
one-dimensional Szász--Mirakjan--Kantorovich theorem, an explicit
observation that the one-dimensional operator is already an Erlang-mixture
operator, and a tensorization argument over the coordinates. A truncation
step then yields genuine finite mixtures, rather than merely countable
ones. This finite-mixture conclusion is the one that is relevant in
practice, since fitting, penalization, and numerical implementation
are carried out in finite-dimensional parameter families; see \cite{KimKottas2022,LeeLin2010,VerbelenAntonioClaeskens2016,YinEtAl2019}.

A second goal is quantitative. For each approximation index $n\ge1$,
the operator $\mathcal{K}_{n}^{(d)}$ produces an Erlang mixture approximation
with common rate $\beta_{n}=n$ and therefore common scale parameter
$q_{n}=1/n$ (using $n$ here in accordance to classical approximation theoretic indexing notation). Using a probabilistic
representation of this operator, we derive compact-set $L^{\infty}$
estimates and, under Hölder assumptions of order $\alpha\in(0,1]$,
local rates of order $n^{-\alpha/2}$. We also obtain global weighted
sup-norm convergence under an intrinsic modulus of continuity condition
adapted to the size of the kernel spread, together with weighted and
unweighted $L^{p}$ rate bounds. In addition, a truncation device
yields rates with respect to the number of mixture components. If
$K$ denotes the total number of components in the approximating finite
mixture, then on fixed compact cubes we obtain a rate of order $K^{-\alpha/(2d)}$.
In global weighted sup norms the same idea still produces the explicit
component-count rate $K^{-\alpha/[2d(2d+\alpha)]}$, and for $1<p<\infty$
we obtain analogous weighted $L^{p}$ component-count rates.

Thus, our contributions are an explicit multivariate $L^{p}$ approximation
theorem for the Erlang mixture class, together with the finite-mixture
and component-count consequences that are relevant in constructive
modelling. In this sense, the present results bring Erlang mixture
approximation closer to the constructive finite-mixture theory available
for location-scale mixtures on $\mathbb{R}^{d}$ (cf. \cite{NguyenEtAl2020,NguyenEtAl2023}).
Further, since the Erlang family is contained in the wider gamma family,
the results immediately imply corresponding denseness statements for
broader classes of product-gamma mixtures; see \cite{BagnatoPunzo2013,BochkinaRousseau2017,WiperInsuaRuggeri2001,YoungEtAl2019}.

The remainder of the paper is organized as follows. Section~\ref{sec:notation}
collects notation, basic definitions, and the one-dimensional approximation
result from \cite{AltomareCappellettiLeonessa2013}, which serves as our only external technical input. Section~\ref{sec:qualitative}
states the qualitative approximation theorems in $L^{p}$ and weighted
sup norms. Section~\ref{sec:scale-rates} reports rates with respect
to the common scale parameter. Section~\ref{sec:component-rates}
presents finite-mixture rates with respect to the number of components.
All proofs and auxiliary technical results are collected in Section~\ref{sec:proofs}.

\section{Notation and technical preliminaries}

\label{sec:notation}

Throughout, $\mathbb{N}=\{1,2,\dots\}$, $\mathbb{N}_{0}=\{0,1,2,\dots\}$,
and $\mathbb{R}_{+}=[0,\infty)$. For $x=(x_{1},\dots,x_{d})\in\mathbb{R}_{+}^{d}$,
we write 
\[
\left\lVert x\right\rVert _{1}=\sum_{j=1}^{d}\left\lvert x_{j}\right\rvert ,\qquad\left\lVert x\right\rVert _{2}=\left(\sum_{j=1}^{d}x_{j}^{2}\right)^{1/2}.
\]
If $j\in\{1,\dots,d\}$, then $x=(x_{j},x_{-j})$ denotes the decomposition
into the $j$th coordinate and the remaining $(d-1)$ coordinates.

\begin{definition} For $d\in\mathbb{N}$ and $1\le p<\infty$, define
\[
\mathcal{D}_{p}(\mathbb{R}_{+}^{d})=\left\{ f\in L^{p}(\mathbb{R}_{+}^{d}):f\ge0\ \text{a.e.\ and }\int_{\mathbb{R}_{+}^{d}}f(x)\,\mathrm{d}x=1\right\}.
\]
Thus $\mathcal{D}_{p}(\mathbb{R}_{+}^{d})$ is the class of probability
densities on $\mathbb{R}_{+}^{d}$ that also belong to $L^{p}(\mathbb{R}_{+}^{d})$.
\end{definition}

For $m\in\mathbb{N}$ and $\beta>0$, let 
\begin{equation}
\tau_{m,\beta}(x)=\frac{\beta(\beta x)^{m-1}e^{-\beta x}}{(m-1)!},\qquad x\ge0.\label{eq:erlang-density-beta}
\end{equation}
This is the Erlang density with integer shape parameter $m$ and rate
$\beta$ (equivalently, scale $1/\beta$). When $\beta=n\in\mathbb{N}$,
we write simply 
\begin{equation}
\tau_{m,n}(x)=ne^{-nx}\frac{(nx)^{m-1}}{(m-1)!},\qquad x\ge0.\label{eq:erlang-density-n}
\end{equation}

\begin{definition} A \emph{finite multivariate Erlang mixture density
with common rate $\beta$} is a density of the form 
\begin{equation}
g(x)=\sum_{m\in F}a_{m}\prod_{j=1}^{d}\tau_{m_{j},\beta}(x_{j}),\qquad x\in\mathbb{R}_{+}^{d},\label{eq:finite-mixed-erlang}
\end{equation}
where $F\subset\mathbb{N}^{d}$ is finite, $a_{m}\ge0$, and $\sum_{m\in F}a_{m}=1$.
\end{definition}

For later use, fix $M>0$ and write $Q_{M}=[0,M]^{d}$. If $r>0$
and $f$ is bounded on $Q_{M+r}$, define the local modulus of continuity
on $Q_{M+r}$ by 
\begin{equation}
\omega_{Q_{M+r}}(f;r)=\sup\left\{\left\lvert f(y)-f(z)\right\rvert :y,z\in Q_{M+r},\ \left\lVert y-z\right\rVert _{2}\le r\right\}.\label{eq:local-modulus}
\end{equation}

For $\nu\ge0$, define the polynomial weight 
\begin{equation}
w_{\nu}(x)=(1+\left\lVert x\right\rVert _{1})^{\nu},\qquad x\in\mathbb{R}_{+}^{d},\label{eq:poly-weight}
\end{equation}
and the corresponding weighted spaces 
\begin{equation}
L_{\infty,\nu}(\mathbb{R}_{+}^{d})=\left\{ f:\mathbb{R}_{+}^{d}\to\mathbb{R}\ \text{measurable}:\left\lVert f\right\rVert _{\infty,\nu}<\infty\right\},\qquad\left\lVert f\right\rVert _{\infty,\nu}=\sup_{x\in\mathbb{R}_{+}^{d}}\frac{\left\lvert f(x)\right\rvert }{w_{\nu}(x)},\label{eq:weighted-sup-space}
\end{equation}
and, for $1\le p<\infty$ and $\eta\ge0$, 
\begin{equation}
L_{p,\eta}(\mathbb{R}_{+}^{d})=L^{p}(\mathbb{R}_{+}^{d},w_{\eta}(x)^{-1}\,\mathrm{d}x),\qquad\left\lVert f\right\rVert _{p,\eta}=\left(\int_{\mathbb{R}_{+}^{d}}\frac{\left\lvert f(x)\right\rvert ^{p}}{w_{\eta}(x)}\,\mathrm{d}x\right)^{1/p}.\label{eq:weighted-lp-space}
\end{equation}
We also write $C_{0}(\mathbb{R}_{+}^{d})$ for the space of bounded
continuous functions on $\mathbb{R}_{+}^{d}$ that vanish at infinity,
that is, satisfy $f(x)\to0$ as $\left\lVert x\right\rVert _{1}\to\infty$.

The global weighted rate statements use an operator-adapted Hölder
seminorm. For $\nu\ge0$ and $0<\alpha\le1$, define 
\begin{equation}
[f]_{\nu,\alpha,*}=\sup\left\{ \frac{\left\lvert f(y)-f(x)\right\rvert }{w_{\nu}(x)\left(\frac{\left\lVert y-x\right\rVert _{2}}{\sqrt{1+\left\lVert x\right\rVert _{1}}}\right)^{\alpha}}:x,y\in\mathbb{R}_{+}^{d},\ y\neq x\right\} .\label{eq:weighted-holder-seminorm}
\end{equation}
We also define the intrinsic weighted modulus of continuity 
\begin{equation}
\Omega_{\nu,*}(f;\delta)=\sup\left\{ \frac{\left\lvert f(y)-f(x)\right\rvert }{w_{\nu}(x)}:x,y\in\mathbb{R}_{+}^{d},\ \left\lVert y-x\right\rVert _{2}\le\delta\sqrt{1+\left\lVert x\right\rVert _{1}}\right\} ,\qquad\delta>0.\label{eq:intrinsic-modulus-def}
\end{equation}
The weight $w_{\nu}$ permits polynomial growth of order $\nu$ at
infinity. The scale $\delta\sqrt{1+\left\lVert x\right\rVert _{1}}$
in $\Omega_{\nu,*}$ is dictated by the operator itself: by Proposition~\ref{prop:prob-representation}
and Lemma~\ref{lem:displacement-moments}, the random displacement
$Y_{n,x}-x$ has second moment of order $(1+\left\lVert x\right\rVert _{1})/n$,
so the kernel centered at $x$ typically spreads over distances of
order $\sqrt{(1+\left\lVert x\right\rVert _{1})/n}$. Thus $\Omega_{\nu,*}$
measures oscillation on the natural spatial scale of the Szász--Mirakjan--Kantorovich
approximation. Weighted moduli of this type are standard in approximation
on unbounded intervals and for positive operators on the half-line;
see, for example, \cite{AltomareCappellettiLeonessa2013,DitzianTotik1987}.

The one-dimensional Szász--Mirakjan--Kantorovich operator is given
by 
\begin{equation}
(K_{n}g)(x)=n\sum_{k=0}^{\infty}e^{-nx}\frac{(nx)^{k}}{k!}\int_{k/n}^{(k+1)/n}g(t)\,\mathrm{d}t,\qquad x\ge0.\label{eq:Kn-def-main}
\end{equation}
The following theorem is the only external technical input that we
use.

\begin{theorem}[One-dimensional $L^p$ convergence; Theorem 3.5 of \cite{AltomareCappellettiLeonessa2013}]\label{thm:1d-input}
Let $1\le p<\infty$. For every $n\in\mathbb{N}$, the operator $K_{n}$
maps $L^{p}(\mathbb{R}_{+})$ into itself, is a positive linear contraction,
and satisfies $\left\lVert K_{n}g\right\rVert _{L^{p}(\mathbb{R}_{+})}\le\left\lVert g\right\rVert _{L^{p}(\mathbb{R}_{+})}$
for every $g\in L^{p}(\mathbb{R}_{+})$. Moreover, $\left\lVert K_{n}g-g\right\rVert _{L^{p}(\mathbb{R}_{+})}\to0$
as $n\to\infty$ for every $g\in L^{p}(\mathbb{R}_{+})$. \end{theorem}

\begin{proposition}[One-dimensional Erlang mixture representation]\label{prop:1d-representation}
Let $1\le p<\infty$ and $f\in\mathcal{D}_{p}(\mathbb{R}_{+})$. For
each $n\in\mathbb{N}$, define 
\begin{equation}
a_{m,n}=\int_{(m-1)/n}^{m/n}f(t)\,\mathrm{d}t,\qquad m\in\mathbb{N}.\label{eq:1d-cell-mass}
\end{equation}
Then $a_{m,n}\ge0$, $\sum_{m=1}^{\infty}a_{m,n}=1$, and 
\begin{equation}
(K_{n}f)(x)=\sum_{m=1}^{\infty}a_{m,n}\,\tau_{m,n}(x),\qquad x\ge0.\label{eq:1d-mixed-erlang-representation}
\end{equation}
Hence $K_{n}f$ is a countable Erlang mixture density with common
scale $1/n$. \end{proposition}

For $d\in\mathbb{N}$ and $j\in\{1,\dots,d\}$, define the coordinatewise
lifted operator 
\begin{equation}
(K_{n,j}f)(x)=n\sum_{k=0}^{\infty}e^{-nx_{j}}\frac{(nx_{j})^{k}}{k!}\int_{k/n}^{(k+1)/n}f(x_{1},\dots,x_{j-1},t,x_{j+1},\dots,x_{d})\,\mathrm{d}t.\label{eq:Knj-def-main}
\end{equation}
We then define the tensorized operator 
\begin{equation}
\mathcal{K}_{n}^{(d)}=K_{n,1}\circ K_{n,2}\circ\cdots\circ K_{n,d}.\label{eq:Kd-def-main}
\end{equation}

The same operator admits a useful probabilistic representation.

\begin{proposition}[A probabilistic representation]\label{prop:prob-representation}
Fix $d\in\mathbb{N}$, $n\in\mathbb{N}$, and $x=(x_{1},\dots,x_{d})\in\mathbb{R}_{+}^{d}$.
Let $N_{1},\dots,N_{d}$ be independent Poisson random variables with
$N_{j}\sim\mathrm{Poisson}(nx_{j})$, and let $U_{1},\dots,U_{d}$
be independent ${\rm Unif}[0,1]$ random variables, independent of
the $N_{j}$. Define $Y_{n,x,j}=(N_{j}+U_{j})/n$ for $j=1,\dots,d$
and write $Y_{n,x}=(Y_{n,x,1},\dots,Y_{n,x,d})$. Then, for every
measurable and locally integrable $f:\mathbb{R}_{+}^{d}\to\mathbb{R}$
such that $\mathbb{E}[|f(Y_{n,x})|]<\infty$,
\begin{equation}
(\mathcal{K}_{n}^{(d)}f)(x)=\mathbb{E}[f(Y_{n,x})].\label{eq:prob-representation-main}
\end{equation}
\end{proposition}

\begin{lemma}[Moments of the displacement]\label{lem:displacement-moments}
Let $\Delta_{n,x}=Y_{n,x}-x$. Then, for every $x\in\mathbb{R}_{+}^{d}$,
\begin{equation}
\mathbb{E}\left[\left\lVert \Delta_{n,x}\right\rVert _{2}^{2}\right]=\frac{\left\lVert x\right\rVert _{1}}{n}+\frac{d}{3n^{2}}.\label{eq:second-moment-displacement-main}
\end{equation}
Consequently, for every $0<r\le2$, 
\begin{equation}
\mathbb{E}\left[\left\lVert \Delta_{n,x}\right\rVert _{2}^{r}\right]\le C_{r,d}\left(\frac{1+\left\lVert x\right\rVert _{1}}{n}\right)^{r/2},\label{eq:r-moment-displacement-main}
\end{equation}
where one may take $C_{r,d}=(1+d/3)^{r/2}$. Finally, for every $\eta>0$,
\begin{equation}
\mathbb{P}\left(\left\lVert \Delta_{n,x}\right\rVert _{2}>\eta\right)\le\frac{1}{\eta^{2}}\left(\frac{\left\lVert x\right\rVert _{1}}{n}+\frac{d}{3n^{2}}\right).\label{eq:tail-displacement-main}
\end{equation}
\end{lemma}

\begin{lemma}[Weighted moments of $Y_{n,x}$]\label{lem:weighted-moments}
Let $\nu\ge0$. Then there exists a constant $A_{\nu,d}>0$ such that
\begin{equation}
\mathbb{E}[w_{\nu}(Y_{n,x})]\le A_{\nu,d}\,w_{\nu}(x)\qquad(x\in\mathbb{R}_{+}^{d},\ n\in\mathbb{N}).\label{eq:weighted-moment-bound-main}
\end{equation}
Consequently, 
\begin{equation}
\left\lVert \mathcal{K}_{n}^{(d)}f\right\rVert _{\infty,\nu}\le A_{\nu,d}\,\left\lVert f\right\rVert _{\infty,\nu}\qquad(f\in L_{\infty,\nu}(\mathbb{R}_{+}^{d}),\ n\in\mathbb{N}).\label{eq:weighted-operator-bound-main}
\end{equation}
\end{lemma}

\begin{remark}
Observe that (\ref{eq:prob-representation-main}) holds for every bounded measurable $f$
and, by Lemma~\ref{lem:weighted-moments}, for every $f\in L_{\infty,\nu}(\mathbb{R}_{+}^{d})$,
$\nu\ge0$. 
\end{remark}

\section{Qualitative approximation theorems}

\label{sec:qualitative}

We begin with the basic $L^{p}$ approximation theorem for the tensorized
operator.

\begin{theorem}[Tensorized $L^p$ convergence]\label{thm:tensor-convergence}
Let $d\in\mathbb{N}$ and $1\le p<\infty$. Then, for every $f\in L^{p}(\mathbb{R}_{+}^{d})$,
\[
\left\lVert \mathcal{K}_{n}^{(d)}f-f\right\rVert _{L^{p}(\mathbb{R}_{+}^{d})}\longrightarrow0\qquad(n\to\infty).
\]
Moreover, for every $n\in\mathbb{N}$, the operator $\mathcal{K}_{n}^{(d)}$
is a positive linear contraction on $L^{p}(\mathbb{R}_{+}^{d})$.
\end{theorem}

The next proposition identifies the tensorized operator with a countable
Erlang mixture density.

\begin{proposition}[Multivariate Erlang mixture formula]\label{prop:explicit-mixed-erlang}
Let $d\in\mathbb{N}$, let $1\le p<\infty$, and let $f\in\mathcal{D}_{p}(\mathbb{R}_{+}^{d})$.
For $m=(m_{1},\dots,m_{d})\in\mathbb{N}^{d}$ and $n\in\mathbb{N}$,
define 
\begin{equation}
Q_{m,n}=\prod_{j=1}^{d}\left[\frac{m_{j}-1}{n},\frac{m_{j}}{n}\right)\label{eq:cells-main}
\end{equation}
and the corresponding cell mass 
\begin{equation}
a_{m,n}=\int_{Q_{m,n}}f(t)\,\mathrm{d}t.\label{eq:cell-masses-main}
\end{equation}
Then $a_{m,n}\ge0$, $\sum_{m\in\mathbb{N}^{d}}a_{m,n}=1$, and 
\begin{equation}
(\mathcal{K}_{n}^{(d)}f)(x)=\sum_{m\in\mathbb{N}^{d}}a_{m,n}\prod_{j=1}^{d}\tau_{m_{j},n}(x_{j}),\qquad x\in\mathbb{R}_{+}^{d}.\label{eq:multivariate-representation-main}
\end{equation}
Consequently, $\mathcal{K}_{n}^{(d)}f$ is a countable multivariate
Erlang mixture density with common rate $\beta_{n}=n$ (equivalently, common
scale $1/n$). \end{proposition}

\begin{corollary}[Countable Erlang mixture approximation]\label{cor:countable}
Let $d\in\mathbb{N}$ and $1\le p<\infty$. For every $f\in\mathcal{D}_{p}(\mathbb{R}_{+}^{d})$,
the sequence $\mathcal{K}_{n}^{(d)}f$ consists of countable multivariate
Erlang mixture densities with common scale parameter $1/n$ and satisfies
\[
\left\lVert \mathcal{K}_{n}^{(d)}f-f\right\rVert _{L^{p}(\mathbb{R}_{+}^{d})}\to0.
\]
\end{corollary}

\begin{corollary}[Finite mixture via compact support]\label{cor:compact-finite}
Let $1\le p<\infty$ and let $f\in\mathcal{D}_{p}(\mathbb{R}_{+}^{d})$
vanish a.e. on $\mathbb{R}_{+}^{d}\setminus Q_{M}$ for some $M>0$. Then,
for every $n\in\mathbb{N}$, the Erlang mixture approximation $\mathcal{K}_{n}^{(d)}f$
is already a finite multivariate Erlang mixture density. More precisely,
only indices $m=(m_{1},\dots,m_{d})$ with $1\le m_{j}\le\lfloor nM\rfloor+1$
can appear in \eqref{eq:multivariate-representation-main}. \end{corollary}

The main qualitative density theorem is the following finite-mixture version.

\begin{theorem}[Finite multivariate Erlang mixture approximation]\label{thm:finite-approx}
Let $d\in\mathbb{N}$ and $1\le p<\infty$. Then the class of finite
multivariate Erlang mixture densities with common scale parameter
is dense in $\mathcal{D}_{p}(\mathbb{R}_{+}^{d})$ with respect to
the $L^{p}(\mathbb{R}_{+}^{d})$ norm. More explicitly, if $f\in\mathcal{D}_{p}(\mathbb{R}_{+}^{d})$,
then there exists a sequence $(g_{n})_{n\ge1}$ of finite multivariate
Erlang mixture densities such that 
\[
\left\lVert g_{n}-f\right\rVert _{L^{p}(\mathbb{R}_{+}^{d})}\longrightarrow0.
\]
Moreover, the common scale parameter of $g_{n}$ may be chosen to
be $1/n$. \end{theorem}

\begin{remark}\label{rem:gamma} Let $d\in\mathbb{N}$ and $1\le p<\infty$.
Any class of densities on $\mathbb{R}_{+}^{d}$ that contains all
finite mixtures of product-Erlang densities is $L^{p}$-dense in $\mathcal{D}_{p}(\mathbb{R}_{+}^{d})$.
In particular, the class of finite mixtures of products of gamma densities
is $L^{p}$-dense in $\mathcal{D}_{p}(\mathbb{R}_{+}^{d})$, because
every Erlang density is a gamma density with integer shape parameter.
\end{remark}

We now consider qualitative approximation in weighted sup norms.

\begin{theorem}[Qualitative convergence in $L_{\infty,\nu}$]\label{thm:weighted-qualitative}
Let $\nu\ge0$ and let $f\in L_{\infty,\nu}(\mathbb{R}_{+}^{d})$.
Assume that 
\begin{equation}
\Omega_{\nu,*}(f;\delta)\longrightarrow0\qquad(\delta\downarrow0).\label{eq:intrinsic-modulus-zero-main}
\end{equation}
Then 
\[
\left\lVert \mathcal{K}_{n}^{(d)}f-f\right\rVert _{\infty,\nu}\longrightarrow0\qquad(n\to\infty).
\]
More quantitatively, for every $\delta>0$ and every $n\in\mathbb{N}$,
\begin{equation}
\left\lVert \mathcal{K}_{n}^{(d)}f-f\right\rVert _{\infty,\nu}\le\Omega_{\nu,*}(f;\delta)+\left\lVert f\right\rVert _{\infty,\nu}\left(\frac{1+d/3}{n\delta^{2}}+\frac{\sqrt{A_{2\nu,d}(1+d/3)}}{\sqrt{n}\,\delta}\right),\label{eq:qualitative-weighted-estimate-main}
\end{equation}
where $A_{2\nu,d}$ is the constant from Lemma~\ref{lem:weighted-moments}
at the index $2\nu$. \end{theorem}

\begin{corollary}[Uniform convergence on $C_{0}$]\label{cor:global-uniform-c0}
Let $f\in C_{0}(\mathbb{R}_{+}^{d})$. Then 
\[
\left\lVert \mathcal{K}_{n}^{(d)}f-f\right\rVert _{L^{\infty}(\mathbb{R}_{+}^{d})}\longrightarrow0\qquad(n\to\infty).
\]
If, in addition, $[f]_{0,\alpha,*}<\infty$ for some $0<\alpha\le1$,
then 
\[
\left\lVert \mathcal{K}_{n}^{(d)}f-f\right\rVert _{L^{\infty}(\mathbb{R}_{+}^{d})}\le C_{\alpha,d}[f]_{0,\alpha,*}\,n^{-\alpha/2}\qquad(n\in\mathbb{N}),
\]
where one may take $C_{\alpha,d}=(1+d/3)^{\alpha/2}$. The same qualitative
conclusion also holds for every bounded continuous $f$ such that
$f-L\in C_{0}(\mathbb{R}_{+}^{d})$ for some constant $L\in\mathbb{R}$.
\end{corollary}

\begin{corollary}[Finite Erlang mixture approximation in weighted
sup norms]\label{cor:weighted-finite} Let $\nu\ge0$, and let $f\in\mathcal{D}_{1}(\mathbb{R}_{+}^{d})\cap L_{\infty,\nu}(\mathbb{R}_{+}^{d})$.
Suppose that $\mathcal{K}_{n}^{(d)}f\to f$ in $L_{\infty,\nu}(\mathbb{R}_{+}^{d})$. Then there exists
a sequence $(g_{n})_{n\ge1}$ of finite multivariate Erlang mixture
densities with common scale parameter $1/n$ such that 
\[
\left\lVert g_{n}-f\right\rVert _{\infty,\nu}\longrightarrow0.
\]
\end{corollary}

\section{Convergence rates in scale}

\label{sec:scale-rates}

We first state the compact-set sup-norm estimate that underlies the
local uniform theory.

\begin{theorem}[Compact-set $L^\infty$ bound]\label{thm:compact-modulus}
Let $M>0$ and let $f:\mathbb{R}_{+}^{d}\to\mathbb{R}$ be bounded
and continuous. Then, for every $r>0$ and every $n\in\mathbb{N}$,
\begin{equation}
\left\lVert \mathcal{K}_{n}^{(d)}f-f\right\rVert _{L^{\infty}(Q_{M})}\le\omega_{Q_{M+r}}(f;r)+\frac{2\left\lVert f\right\rVert _{L^{\infty}(\mathbb{R}_{+}^{d})}}{r^{2}}\left(\frac{dM}{n}+\frac{d}{3n^{2}}\right).\label{eq:compact-modulus-estimate-main}
\end{equation}
\end{theorem}

\begin{corollary}[Local uniform convergence on compact sets]\label{cor:compact-uniform}
If $f$ is bounded and continuous on $\mathbb{R}_{+}^{d}$, then for
every $M>0$, 
\[
\left\lVert \mathcal{K}_{n}^{(d)}f-f\right\rVert _{L^{\infty}(Q_{M})}\longrightarrow0\qquad(n\to\infty).
\]
\end{corollary}

\begin{theorem}[Compact-set H\"older rate]\label{thm:compact-holder}
Let $M>0$, let $0<\alpha\le1$, and let $f:\mathbb{R}_{+}^{d}\to\mathbb{R}$
be bounded and measurable. Assume that there is a constant $H>0$ such that 
\begin{equation}
\left\lvert f(y)-f(z)\right\rvert \le H\left\lVert y-z\right\rVert _{2}^{\alpha}\qquad(y,z\in Q_{M+1}).\label{eq:local-holder-compact-main}
\end{equation}
Then, for every $n\in\mathbb{N}$, 
\begin{equation}
\left\lVert \mathcal{K}_{n}^{(d)}f-f\right\rVert _{L^{\infty}(Q_{M})}\le HC_{\alpha,d}\left(\frac{1+dM}{n}\right)^{\alpha/2}+2\left\lVert f\right\rVert _{L^{\infty}(\mathbb{R}_{+}^{d})}\left(\frac{dM}{n}+\frac{d}{3n^{2}}\right),\label{eq:compact-holder-rate-main}
\end{equation}
where one may take $C_{\alpha,d}=(1+d/3)^{\alpha/2}$. In particular,
\[
\left\lVert \mathcal{K}_{n}^{(d)}f-f\right\rVert _{L^{\infty}(Q_{M})}=O(n^{-\alpha/2}).
\]
\end{theorem}

The next pair of theorems give global weighted sup-norm and $L^p$ rates.

\begin{theorem}[Weighted H\"older rate in $L_{\infty,\nu}$]\label{thm:weighted-holder-rate}
Let $\nu\ge0$ and $0<\alpha\le1$. If $f\in L_{\infty,\nu}(\mathbb{R}_{+}^{d})$
and $[f]_{\nu,\alpha,*}<\infty$, then for every $n\in\mathbb{N}$,
\begin{equation}
\left\lVert \mathcal{K}_{n}^{(d)}f-f\right\rVert _{\infty,\nu}\le C_{\alpha,d}[f]_{\nu,\alpha,*}\,n^{-\alpha/2},\label{eq:weighted-holder-rate-main}
\end{equation}
where one may take $C_{\alpha,d}=(1+d/3)^{\alpha/2}$. \end{theorem}

\begin{theorem}[Weighted $L^p$ rate from an integrated intrinsic H\"older envelope]\label{thm:weighted-lp-rate}
Let $1\le p<\infty$, let $\eta\ge0$, let $0<\alpha\le1$, and let
$f:\mathbb{R}_{+}^{d}\to\mathbb{R}$ be measurable. Assume that there
exists a measurable function $H\in L_{p,\eta}(\mathbb{R}_{+}^{d})$
such that 
\begin{equation}
\left\lvert f(y)-f(x)\right\rvert \le H(x)\left(\frac{\left\lVert y-x\right\rVert _{2}}{\sqrt{1+\left\lVert x\right\rVert _{1}}}\right)^{\alpha}\qquad(x,y\in\mathbb{R}_{+}^{d}).\label{eq:weighted-lp-envelope-main}
\end{equation}
Then, for every $n\in\mathbb{N}$, 
\begin{equation}
\left\lVert \mathcal{K}_{n}^{(d)}f-f\right\rVert _{p,\eta}\le C_{\alpha,d}\,n^{-\alpha/2}\,\left\lVert H\right\rVert _{p,\eta},\label{eq:weighted-lp-rate-main}
\end{equation}
where one may take $C_{\alpha,d}=(1+d/3)^{\alpha/2}$. \end{theorem}

\begin{corollary}[From weighted sup-norm control to weighted $L^{p}$
control]\label{cor:weighted-sup-to-lp} Let $1\le p<\infty$, let
$\nu\ge0$, let $0<\alpha\le1$, and let $\eta>\nu p+d$. Suppose
that $f\in L_{\infty,\nu}(\mathbb{R}_{+}^{d})$ and $[f]_{\nu,\alpha,*}<\infty$.
Then 
\begin{equation}
\left\lVert \mathcal{K}_{n}^{(d)}f-f\right\rVert _{p,\eta}\le C_{\alpha,d}\,B_{p,\eta,\nu,d}\,[f]_{\nu,\alpha,*}\,n^{-\alpha/2},\label{eq:weighted-sup-to-lp-main}
\end{equation}
where 
\[
B_{p,\eta,\nu,d}=\left(\int_{\mathbb{R}_{+}^{d}}(1+\left\lVert x\right\rVert _{1})^{\nu p-\eta}\,\mathrm{d}x\right)^{1/p}<\infty.
\]
In particular, $\left\lVert \mathcal{K}_{n}^{(d)}f-f\right\rVert _{p,\eta}=O(n^{-\alpha/2})$.
\end{corollary}

\begin{remark}[Unweighted $L^{p}$ rates]\label{rem:unweighted-lp}
Taking $\eta=0$ in Theorem~\ref{thm:weighted-lp-rate} yields the
ordinary unweighted estimate 
\[
\left\lVert \mathcal{K}_{n}^{(d)}f-f\right\rVert _{L^{p}(\mathbb{R}_{+}^{d})}\le C_{\alpha,d}\,n^{-\alpha/2}\,\left\lVert H\right\rVert _{L^{p}(\mathbb{R}_{+}^{d})}
\]
whenever $H\in L^{p}(\mathbb{R}_{+}^{d})$ satisfies \eqref{eq:weighted-lp-envelope-main}.
Thus the probabilistic argument gives unweighted $L^{p}$ rates. However,
this is weaker than the one-dimensional theory of \cite{Totik1983,Totik1985}. A correspondingly sharp multivariate theorem would require
a more nuanced approach. \end{remark}

\section{Convergence rates in the number of components}

\label{sec:component-rates}

We now pass from scale-dependent rates for the countable Erlang mixture
operator to explicit rates for finite mixtures indexed by the number
of components.

\begin{proposition}[Compact-set truncation]\label{prop:compact-tailfree}
Let $d\in\mathbb{N}$, let $M>0$, and let $f\in\mathcal{D}_{1}(\mathbb{R}_{+}^{d})$.
Write 
\[
\mathcal{K}_{n}^{(d)}f=\sum_{m\in\mathbb{N}^{d}}a_{m,n}\,\varphi_{m,n},\qquad\varphi_{m,n}(x)=\prod_{j=1}^{d}\tau_{m_{j},n}(x_{j}).
\]
For $N\in\mathbb{N}$ with $N\ge nM$, define 
\[
F_{N}=\{m\in\mathbb{N}^{d}:1\le m_{j}\le N\text{ for all }j=1,\dots,d\},\qquad r_{n,N}=\sum_{m\notin F_{N}}a_{m,n},
\]
set $\ell_{N}=(N+1,\dots,N+1)\in\mathbb{N}^{d}$, and define 
\[
\widetilde{g}_{n,N}^{(M)}=\sum_{m\in F_{N}}a_{m,n}\,\varphi_{m,n}+r_{n,N}\,\varphi_{\ell_{N},n}.
\]
Then $\widetilde{g}_{n,N}^{(M)}$ is a finite multivariate Erlang
mixture density with at most $N^{d}+1$ components, and 
\begin{equation}
\left\lVert \mathcal{K}_{n}^{(d)}f-\widetilde{g}_{n,N}^{(M)}\right\rVert _{L^{\infty}(Q_{M})}\le2n^{d}e^{-nM}\frac{(nM)^{N}}{N!}.\label{eq:compact-tailfree-trunc-main}
\end{equation}
If, in addition, $f$ is bounded and continuous on $\mathbb{R}_{+}^{d}$,
then for every $r>0$, 
\begin{align}
\left\lVert \widetilde{g}_{n,N}^{(M)}-f\right\rVert _{L^{\infty}(Q_{M})} & \le\omega_{Q_{M+r}}(f;r)+\frac{2\left\lVert f\right\rVert _{L^{\infty}(\mathbb{R}_{+}^{d})}}{r^{2}}\left(\frac{dM}{n}+\frac{d}{3n^{2}}\right)\nonumber \\
 & \qquad+2n^{d}e^{-nM}\frac{(nM)^{N}}{N!}.\label{eq:compact-tailfree-modulus-main}
\end{align}
\end{proposition}

\begin{corollary}[Component-count compact-set rate]\label{cor:compact-component-rate}
Let $d\in\mathbb{N}$, let $M>0$, let $0<\alpha\le1$, and let $f\in\mathcal{D}_{1}(\mathbb{R}_{+}^{d})\cap L^{\infty}(\mathbb{R}_{+}^{d})$.
Assume that there is a constant $H>0$ such that 
\[
\left\lvert f(y)-f(z)\right\rvert \le H\left\lVert y-z\right\rVert _{2}^{\alpha}\qquad(y,z\in Q_{M+1}).
\]
Then there exist constants $C>0$ and $K_{0}\in\mathbb{N}$ such that
for every integer $K\ge K_{0}$ there exists a finite multivariate
Erlang mixture density $g_{K}$ with at most $K$ components and 
\[
\left\lVert g_{K}-f\right\rVert _{L^{\infty}(Q_{M})}\le CK^{-\alpha/(2d)}.
\]
\end{corollary}

To obtain a global weighted sup-norm component-count rate, we require the following
shape-decay bound for Erlang kernels.

\begin{lemma}[Shape decay of Erlang kernels in sup norm]\label{lem:erlang-shape-sup}
For every $m,n\in\mathbb{N}$, 
\[
\left\lVert \tau_{m,n}\right\rVert _{L^{\infty}(\mathbb{R}_{+})}\le nm^{-1/2}.
\]
Consequently, if $m=(m_{1},\dots,m_{d})\in\mathbb{N}^{d}$ has some
coordinate $m_{j}\ge N+1$, then for every $\nu\ge0$, 
\[
\left\lVert \varphi_{m,n}\right\rVert _{\infty,\nu}\le n^{d}(N+1)^{-1/2}.
\]
\end{lemma}

\begin{proposition}[Truncation in weighted sup norms]\label{prop:weighted-tailfree}
Let $d\in\mathbb{N}$, let $\nu\ge0$, and let $f\in\mathcal{D}_{1}(\mathbb{R}_{+}^{d})$.
Write 
\[
\mathcal{K}_{n}^{(d)}f=\sum_{m\in\mathbb{N}^{d}}a_{m,n}\,\varphi_{m,n}.
\]
For $N\in\mathbb{N}$, define 
\[
F_{N}=\{m\in\mathbb{N}^{d}:1\le m_{j}\le N\text{ for all }j=1,\dots,d\},\qquad r_{n,N}=\sum_{m\notin F_{N}}a_{m,n},
\]
set $\ell_{N}=(N+1,\dots,N+1)$, and define 
\[
\widetilde{g}_{n,N}^{(\nu)}=\sum_{m\in F_{N}}a_{m,n}\,\varphi_{m,n}+r_{n,N}\,\varphi_{\ell_{N},n}.
\]
Then $\widetilde{g}_{n,N}^{(\nu)}$ is a finite multivariate Erlang
mixture density with at most $N^{d}+1$ components, and 
\begin{equation}
\left\lVert \mathcal{K}_{n}^{(d)}f-\widetilde{g}_{n,N}^{(\nu)}\right\rVert _{\infty,\nu}\le2n^{d}(N+1)^{-1/2}.\label{eq:weighted-tailfree-trunc-main}
\end{equation}
\end{proposition}

\begin{corollary}[Component-count rate for weighted sup norms]\label{cor:weighted-component-rate}
Let $d\in\mathbb{N}$, let $\nu\ge0$, let $0<\alpha\le1$, and let
$f\in\mathcal{D}_{1}(\mathbb{R}_{+}^{d})\cap L_{\infty,\nu}(\mathbb{R}_{+}^{d})$.
Assume that $[f]_{\nu,\alpha,*}<\infty$. Then there exist constants
$C>0$ and $K_{0}\in\mathbb{N}$ such that for every integer $K\ge K_{0}$
there exists a finite multivariate Erlang mixture density $g_{K}$
with at most $K$ components and 
\[
\left\lVert g_{K}-f\right\rVert _{\infty,\nu}\le CK^{-\alpha/[2d(2d+\alpha)]}.
\]
 \end{corollary}

The same truncation idea also yields component-count
rates in weighted $L^{p}$ norms for $1<p<\infty$.

\begin{lemma}[Shape decay of Erlang kernels in $L^{p}$]\label{lem:erlang-shape-lp}
Let $1<p<\infty$. Then, for every $m,n\in\mathbb{N}$, 
\[
\left\lVert \tau_{m,n}\right\rVert _{L^{p}(\mathbb{R}_{+})}\le n^{1-1/p}m^{-(1-1/p)/2}.
\]
Consequently, if $m=(m_{1},\dots,m_{d})\in\mathbb{N}^{d}$ has some
coordinate $m_{j}\ge N+1$, then for every $\eta\ge0$, 
\[
\left\lVert \varphi_{m,n}\right\rVert _{p,\eta}\le n^{d(1-1/p)}(N+1)^{-(1-1/p)/2}.
\]
\end{lemma}

\begin{proposition}[Truncation in weighted $L^{p}$ norms]\label{prop:weighted-lp-tailfree}
Let $d\in\mathbb{N}$, let $1<p<\infty$, let $\eta\ge0$, and let
$f\in\mathcal{D}_{1}(\mathbb{R}_{+}^{d})$. Write 
\[
\mathcal{K}_{n}^{(d)}f=\sum_{m\in\mathbb{N}^{d}}a_{m,n}\,\varphi_{m,n}.
\]
For $N\in\mathbb{N}$, define 
\[
F_{N}=\{m\in\mathbb{N}^{d}:1\le m_{j}\le N\text{ for all }j=1,\dots,d\},\qquad r_{n,N}=\sum_{m\notin F_{N}}a_{m,n},
\]
set $\ell_{N}=(N+1,\dots,N+1)$, and define 
\[
\widetilde{g}_{n,N}^{(p,\eta)}=\sum_{m\in F_{N}}a_{m,n}\,\varphi_{m,n}+r_{n,N}\,\varphi_{\ell_{N},n}.
\]
Then $\widetilde{g}_{n,N}^{(p,\eta)}$ is a finite multivariate Erlang
mixture density with at most $N^{d}+1$ components, and 
\begin{equation}
\left\lVert \mathcal{K}_{n}^{(d)}f-\widetilde{g}_{n,N}^{(p,\eta)}\right\rVert _{p,\eta}\le2n^{d(1-1/p)}(N+1)^{-(1-1/p)/2}.\label{eq:weighted-lp-tailfree-trunc-main}
\end{equation}
\end{proposition}

\begin{corollary}[Component-count rate in weighted $L^{p}$ norms]\label{cor:weighted-lp-component-rate}
Let $d\in\mathbb{N}$, let $1<p<\infty$, let $\eta\ge0$, let $0<\alpha\le1$,
and let $f\in\mathcal{D}_{1}(\mathbb{R}_{+}^{d})$. Assume that there
exists a measurable function $H\in L_{p,\eta}(\mathbb{R}_{+}^{d})$
such that 
\[
\left\lvert f(y)-f(x)\right\rvert \le H(x)\left(\frac{\left\lVert y-x\right\rVert _{2}}{\sqrt{1+\left\lVert x\right\rVert _{1}}}\right)^{\alpha}\qquad(x,y\in\mathbb{R}_{+}^{d}).
\]
Then there exist constants $C>0$ and $K_{0}\in\mathbb{N}$ such that
for every integer $K\ge K_{0}$ there exists a finite multivariate
Erlang mixture density $g_{K}$ with at most $K$ components and 
\[
\left\lVert g_{K}-f\right\rVert _{p,\eta}\le CK^{-\alpha/[2d(2d+\alpha p/(p-1))]}.
\]
\end{corollary}

\begin{remark}\label{rem:component-rate-remarks} The compact-set
exponent in Corollary~\ref{cor:compact-component-rate} is faster
because kernels whose modes lie beyond a fixed compact cube become
exponentially small on that cube. In the global weighted theory we
only use algebraic shape decay, and the resulting component-count
exponents are correspondingly slower. Corollary~\ref{cor:weighted-component-rate}
may be viewed as the case $p=\infty$ of Corollary~\ref{cor:weighted-lp-component-rate}.
By contrast, the present truncation strategy does not yield a component-count
rate in ordinary unweighted $L^{1}$, because every Erlang density
has $L^{1}$ norm equal to $1$, independently of its shape. \end{remark}

\section{Proofs and technical results}

\label{sec:proofs}

\subsection{One-dimensional input, tensorization, and qualitative approximation}

The next two lemmas transfer the one-dimensional $L^{p}$ approximation
property to the coordinatewise lifted operators.

\begin{lemma}[Coordinatewise $L^{p}$ contraction]\label{lem:coord-contraction}
Let $1\le p<\infty$, $n\in\mathbb{N}$, and $j\in\{1,\dots,d\}$.
Then $K_{n,j}$ maps $L^{p}(\mathbb{R}_{+}^{d})$ into itself and
\[
\left\lVert K_{n,j}f\right\rVert _{L^{p}(\mathbb{R}_{+}^{d})}\le\left\lVert f\right\rVert _{L^{p}(\mathbb{R}_{+}^{d})}\qquad(f\in L^{p}(\mathbb{R}_{+}^{d})).
\]
In particular, each $K_{n,j}$ is a positive linear contraction on
$L^{p}(\mathbb{R}_{+}^{d})$. \end{lemma}

\begin{lemma}[Coordinatewise convergence]\label{lem:coord-convergence}
Let $1\le p<\infty$ and $j\in\{1,\dots,d\}$. Then, for every $f\in L^{p}(\mathbb{R}_{+}^{d})$,
\[
\left\lVert K_{n,j}f-f\right\rVert _{L^{p}(\mathbb{R}_{+}^{d})}\longrightarrow0\qquad(n\to\infty).
\]
\end{lemma}

\begin{lemma}[Uniform $L^{p}$ bound for Erlang kernels]\label{lem:kernel-bound}
Let $1\le p<\infty$ and $m,n\in\mathbb{N}$. Then
\[
0\le\tau_{m,n}(x)\le n\qquad(x\ge0),
\]
and therefore 
\[
\left\lVert \tau_{m,n}\right\rVert _{L^{p}(\mathbb{R}_{+})}\le n^{1-1/p}.
\]
Consequently, for every $m=(m_{1},\dots,m_{d})\in\mathbb{N}^{d}$,
\begin{equation}
\left\|\prod_{j=1}^{d}\tau_{m_{j},n}(\cdot_{j})\right\|_{L^{p}(\mathbb{R}_{+}^{d})}\le n^{d(1-1/p)}.\label{eq:product-bound-main}
\end{equation}
\end{lemma} 
\begin{proof}[Proof of Theorem~\ref{thm:1d-input}]
The positivity and linearity are immediate from \eqref{eq:Kn-def-main}.
For the contraction property, write 
\[
p_{n,k}(x)=e^{-nx}\frac{(nx)^{k}}{k!},\qquad I_{k}=\left[\frac{k}{n},\frac{k+1}{n}\right),\qquad k\in\mathbb{N}_{0}.
\]
Since $\sum_{k=0}^{\infty}p_{n,k}(x)=1$, Jensen's inequality and then
Jensen's inequality on each interval $I_{k}$ with respect to the probability
measure $n\,\mathrm{d}t$ give 
\begin{align*}
|K_{n}g(x)|^{p} & \le\sum_{k=0}^{\infty}p_{n,k}(x)\left(n\int_{I_{k}}|g(t)|\,\mathrm{d}t\right)^{p}\\
 & \le\sum_{k=0}^{\infty}p_{n,k}(x)\,n\int_{I_{k}}|g(t)|^{p}\,\mathrm{d}t.
\end{align*}
Integrating over $x\in\mathbb{R}_{+}$ and using Tonelli's theorem, we obtain
\begin{align*}
\|K_{n}g\|_{L^{p}(\mathbb{R}_{+})}^{p} & \le\sum_{k=0}^{\infty}\left(\int_{0}^{\infty}p_{n,k}(x)\,\mathrm{d}x\right)\,n\int_{I_{k}}|g(t)|^{p}\,\mathrm{d}t\\
 & =\sum_{k=0}^{\infty}\int_{I_{k}}|g(t)|^{p}\,\mathrm{d}t=\|g\|_{L^{p}(\mathbb{R}_{+})}^{p},
\end{align*}
because $\int_{0}^{\infty}p_{n,k}(x)\,\mathrm{d}x=1/n$. Thus $K_{n}$ is a positive
linear contraction on $L^{p}(\mathbb{R}_{+})$. The remaining strong $L^{p}$
convergence assertion is exactly Theorem~3.5 of
\cite{AltomareCappellettiLeonessa2013}, specialized to the choice
$a_{n}=0$ and $b_{n}=1$, for which the generalized operators in that
paper coincide with $K_{n}$. This proves the theorem. 
\end{proof}
\begin{proof}[Proof of Proposition~\ref{prop:1d-representation}]
The positivity of the coefficients in \eqref{eq:1d-cell-mass} is
immediate. Since the intervals $[(m-1)/n,m/n)$ partition $\mathbb{R}_{+}$
up to endpoints, 
\[
\sum_{m=1}^{\infty}a_{m,n}=\sum_{m=1}^{\infty}\int_{(m-1)/n}^{m/n}f(t)\,\mathrm{d}t=\int_{0}^{\infty}f(t)\,\mathrm{d}t=1.
\]
Re-indexing \eqref{eq:Kn-def-main} with $m=k+1$ gives 
\begin{align*}
(K_{n}f)(x) & =n\sum_{k=0}^{\infty}e^{-nx}\frac{(nx)^{k}}{k!}\int_{k/n}^{(k+1)/n}f(t)\,\mathrm{d}t\\
 & =\sum_{m=1}^{\infty}\left(\int_{(m-1)/n}^{m/n}f(t)\,\mathrm{d}t\right)ne^{-nx}\frac{(nx)^{m-1}}{(m-1)!}\\
 & =\sum_{m=1}^{\infty}a_{m,n}\,\tau_{m,n}(x),
\end{align*}
which is exactly \eqref{eq:1d-mixed-erlang-representation}. 
\end{proof}
\begin{proof}[Proof of Lemma~\ref{lem:coord-contraction}]
Fix $f\in L^{p}(\mathbb{R}_{+}^{d})$. By Fubini's theorem, for almost
every $x_{-j}\in\mathbb{R}_{+}^{d-1}$ the slice $f_{x_{-j}}(t)=f(x_{1},\dots,x_{j-1},t,x_{j+1},\dots,x_{d})$
belongs to $L^{p}(\mathbb{R}_{+})$. For such $x_{-j}$ we have $(K_{n,j}f)_{x_{-j}}=K_{n}(f_{x_{-j}})$.
Theorem~\ref{thm:1d-input} therefore gives 
\[
\int_{0}^{\infty}\left\lvert (K_{n,j}f)(x_{j},x_{-j})\right\rvert ^{p}\,\mathrm{d}x_{j}=\left\lVert K_{n}(f_{x_{-j}})\right\rVert _{L^{p}(\mathbb{R}_{+})}^{p}\le\left\lVert f_{x_{-j}}\right\rVert _{L^{p}(\mathbb{R}_{+})}^{p}.
\]
Integrating with respect to $x_{-j}$ and using Fubini's theorem once more yields
\begin{align*}
\left\lVert K_{n,j}f\right\rVert _{L^{p}(\mathbb{R}_{+}^{d})}^{p} & =\int_{\mathbb{R}_{+}^{d-1}}\int_{0}^{\infty}\left\lvert (K_{n,j}f)(x_{j},x_{-j})\right\rvert ^{p}\,\mathrm{d}x_{j}\,\mathrm{d}x_{-j}\\
 & \le\int_{\mathbb{R}_{+}^{d-1}}\int_{0}^{\infty}\left\lvert f(x_{j},x_{-j})\right\rvert ^{p}\,\mathrm{d}x_{j}\,\mathrm{d}x_{-j}=\left\lVert f\right\rVert _{L^{p}(\mathbb{R}_{+}^{d})}^{p}.
\end{align*}
The positivity and linearity are immediate from \eqref{eq:Knj-def-main}. 
\end{proof}
\begin{proof}[Proof of Lemma~\ref{lem:coord-convergence}]
Fix $f\in L^{p}(\mathbb{R}_{+}^{d})$ and $j\in\{1,\dots,d\}$. For
almost every $x_{-j}$ we again have $f_{x_{-j}}\in L^{p}(\mathbb{R}_{+})$
and $(K_{n,j}f)_{x_{-j}}=K_{n}(f_{x_{-j}})$. Therefore, by Theorem~\ref{thm:1d-input},
\[
\left\lVert K_{n}(f_{x_{-j}})-f_{x_{-j}}\right\rVert _{L^{p}(\mathbb{R}_{+})}\longrightarrow0\qquad(n\to\infty)
\]
for almost every $x_{-j}$. To pass to the outer integral, note that
the contraction property gives 
\begin{align*}
\left\lVert K_{n}(f_{x_{-j}})-f_{x_{-j}}\right\rVert _{L^{p}(\mathbb{R}_{+})}^{p} & \le2^{p-1}\left(\left\lVert K_{n}(f_{x_{-j}})\right\rVert _{L^{p}(\mathbb{R}_{+})}^{p}+\left\lVert f_{x_{-j}}\right\rVert _{L^{p}(\mathbb{R}_{+})}^{p}\right)\\
 & \le2^{p}\left\lVert f_{x_{-j}}\right\rVert _{L^{p}(\mathbb{R}_{+})}^{p}.
\end{align*}
The function $x_{-j}\mapsto\left\lVert f_{x_{-j}}\right\rVert _{L^{p}(\mathbb{R}_{+})}^{p}$
is integrable by Fubini's theorem, so dominated convergence yields 
\[
\left\lVert K_{n,j}f-f\right\rVert _{L^{p}(\mathbb{R}_{+}^{d})}^{p}=\int_{\mathbb{R}_{+}^{d-1}}\left\lVert K_{n}(f_{x_{-j}})-f_{x_{-j}}\right\rVert _{L^{p}(\mathbb{R}_{+})}^{p}\,\mathrm{d}x_{-j}\longrightarrow0.
\]
\end{proof}
\begin{proof}[Proof of Theorem~\ref{thm:tensor-convergence}]
By Lemma~\ref{lem:coord-contraction}, each $K_{n,j}$ is a positive
linear contraction on $L^{p}(\mathbb{R}_{+}^{d})$, hence so is their
composition $\mathcal{K}_{n}^{(d)}$. For convergence, write the telescoping
identity 
\begin{align*}
\mathcal{K}_{n}^{(d)}-I & =K_{n,1}\cdots K_{n,d}-I\\
 & =K_{n,1}\cdots K_{n,d-1}(K_{n,d}-I)+K_{n,1}\cdots K_{n,d-2}(K_{n,d-1}-I)+\cdots+(K_{n,1}-I).
\end{align*}
Applying this to $f$ and using that every preceding product is a
contraction, we obtain 
\begin{align*}
\left\lVert \mathcal{K}_{n}^{(d)}f-f\right\rVert _{L^{p}} & \le\sum_{j=1}^{d}\left\lVert K_{n,j}f-f\right\rVert _{L^{p}}.
\end{align*}
Each term on the right tends to $0$ by Lemma~\ref{lem:coord-convergence},
and the result follows. 
\end{proof}
\begin{proof}[Proof of Proposition~\ref{prop:explicit-mixed-erlang}]
The coefficients in \eqref{eq:cell-masses-main} are nonnegative
because $f\ge0$ a.e. Since the rectangles $Q_{m,n}$ partition $\mathbb{R}_{+}^{d}$
up to boundaries, $\sum_{m\in\mathbb{N}^{d}}a_{m,n}=\sum_{m\in\mathbb{N}^{d}}\int_{Q_{m,n}}f(t)\,\mathrm{d}t=\int_{\mathbb{R}_{+}^{d}}f(t)\,\mathrm{d}t=1$.
A repeated application of Tonelli's theorem shows that 
\begin{align*}
(\mathcal{K}_{n}^{(d)}f)(x) & =n^{d}\sum_{k\in\mathbb{N}_{0}^{d}}\left(\prod_{j=1}^{d}e^{-nx_{j}}\frac{(nx_{j})^{k_{j}}}{k_{j}!}\right)\int_{\prod_{j=1}^{d}[k_{j}/n,(k_{j}+1)/n)}f(t)\,\mathrm{d}t.
\end{align*}
Substituting $m_{j}=k_{j}+1$ for each coordinate transforms the cell
into $Q_{m,n}$ and the product factor into $n^{d}\prod_{j=1}^{d}e^{-nx_{j}}\frac{(nx_{j})^{k_{j}}}{k_{j}!}=\prod_{j=1}^{d}ne^{-nx_{j}}\frac{(nx_{j})^{m_{j}-1}}{(m_{j}-1)!}=\prod_{j=1}^{d}\tau_{m_{j},n}(x_{j})$.
This proves \eqref{eq:multivariate-representation-main}. 
\end{proof}
\begin{proof}[Proof of Corollary~\ref{cor:countable}]
Combine Theorem~\ref{thm:tensor-convergence} with Proposition~\ref{prop:explicit-mixed-erlang}. 
\end{proof}
\begin{proof}[Proof of Corollary~\ref{cor:compact-finite}]
If some coordinate $m_{j}$ satisfies $m_{j}\ge\lfloor nM\rfloor+2$,
then the interval $[(m_{j}-1)/n,m_{j}/n)$ lies to the right of $M$,
so the cell $Q_{m,n}$ is disjoint from $Q_{M}$ in the $j$th coordinate.
Since $f$ vanishes a.e.\ outside $Q_{M}$, the corresponding coefficient
$a_{m,n}$ is zero. Therefore only finitely many indices can appear
in \eqref{eq:multivariate-representation-main}. 
\end{proof}
\begin{proof}[Proof of Lemma~\ref{lem:kernel-bound}]
Since $e^{y}=\sum_{r=0}^{\infty}y^{r}/r!\ge y^{m-1}/(m-1)!$ for
$y\ge0$, we have $e^{-y}y^{m-1}/(m-1)!\le1$ for all $y\ge0$. Taking
$y=nx$ gives $\tau_{m,n}(x)\le n$, and therefore $\left\lVert \tau_{m,n}\right\rVert _{L^{p}(\mathbb{R}_{+})}^{p}=\int_{0}^{\infty}\tau_{m,n}(x)^{p}\,\mathrm{d}x\le n^{p-1}\int_{0}^{\infty}\tau_{m,n}(x)\,\mathrm{d}x=n^{p-1}$,
because $\tau_{m,n}$ is a probability density. This proves the one-dimensional
estimate. The product bound \eqref{eq:product-bound-main} follows
by Fubini's theorem. 
\end{proof}
\begin{proof}[Proof of Theorem~\ref{thm:finite-approx}]
Fix $f\in\mathcal{D}_{p}(\mathbb{R}_{+}^{d})$ and write $\mathcal{K}_{n}^{(d)}f=\sum_{m\in\mathbb{N}^{d}}a_{m,n}\,\varphi_{m,n}$,
where $\varphi_{m,n}(x)=\prod_{j=1}^{d}\tau_{m_{j},n}(x_{j})$. For
$N\in\mathbb{N}$, define the finite index set $F_{N}=\{m\in\mathbb{N}^{d}:1\le m_{j}\le N\text{ for all }j=1,\dots,d\}$
and the tail mass $r_{n,N}=\sum_{m\notin F_{N}}a_{m,n}$. Because
the coefficients are nonnegative and sum to $1$, we have $r_{n,N}\downarrow0$
as $N\to\infty$ for each fixed $n$. Define the finite mixture 
\begin{equation}
g_{n,N}=\sum_{m\in F_{N}}a_{m,n}\,\varphi_{m,n}+r_{n,N}\,\varphi_{\mathbf{1},n},\label{eq:gnN-main}
\end{equation}
where $\mathbf{1}=(1,\dots,1)\in\mathbb{N}^{d}$. Then $g_{n,N}$
is a finite multivariate Erlang mixture density.

Since $\mathcal{K}_{n}^{(d)}f-g_{n,N}=\sum_{m\notin F_{N}}a_{m,n}(\varphi_{m,n}-\varphi_{\mathbf{1},n})$,
Lemma~\ref{lem:kernel-bound} gives 
\begin{align*}
\left\lVert \mathcal{K}_{n}^{(d)}f-g_{n,N}\right\rVert _{L^{p}} & \le\sum_{m\notin F_{N}}a_{m,n}\left(\left\lVert \varphi_{m,n}\right\rVert _{L^{p}}+\left\lVert \varphi_{\mathbf{1},n}\right\rVert _{L^{p}}\right)\\
 & \le2n^{d(1-1/p)}\sum_{m\notin F_{N}}a_{m,n}=2n^{d(1-1/p)}r_{n,N}.
\end{align*}
For each fixed $n$, choose $N(n)$ so large that $\left\lVert \mathcal{K}_{n}^{(d)}f-g_{n,N(n)}\right\rVert _{L^{p}(\mathbb{R}_{+}^{d})}\le1/n$.
Set $g_{n}=g_{n,N(n)}$. Then each $g_{n}$ is a finite multivariate
Erlang mixture density with common scale $1/n$, and 
\begin{align*}
\left\lVert g_{n}-f\right\rVert _{L^{p}} & \le\left\lVert g_{n}-\mathcal{K}_{n}^{(d)}f\right\rVert _{L^{p}}+\left\lVert \mathcal{K}_{n}^{(d)}f-f\right\rVert _{L^{p}}\\
 & \le\frac{1}{n}+\left\lVert \mathcal{K}_{n}^{(d)}f-f\right\rVert _{L^{p}}\longrightarrow0
\end{align*}
by Theorem~\ref{thm:tensor-convergence}. 
\end{proof}
\begin{proof}[Proof of Theorem~\ref{thm:weighted-qualitative}]
Fix $x\in\mathbb{R}_{+}^{d}$ and write $Y=Y_{n,x}$ and $\Delta=Y-x$.
Let $A_{n,\delta}(x)=\{\left\lVert \Delta\right\rVert _{2}>\delta\sqrt{1+\left\lVert x\right\rVert _{1}}\}$.
Since $f\in L_{\infty,\nu}(\mathbb{R}_{+}^{d})$ and Lemma~\ref{lem:weighted-moments}
gives $\mathbb{E}[w_{\nu}(Y)]<\infty$, we have $\mathbb{E}[|f(Y)|]<\infty$.
Hence Proposition~\ref{prop:prob-representation} applies and $(\mathcal{K}_{n}^{(d)}f)(x)-f(x)=\mathbb{E}[f(Y)-f(x)]$.
On the near event $A_{n,\delta}(x)^{c}$, the definition of $\Omega_{\nu,*}(f;\delta)$
gives $\left\lvert f(Y)-f(x)\right\rvert /w_{\nu}(x)\le\Omega_{\nu,*}(f;\delta)$,
and therefore $\frac{1}{w_{\nu}(x)}\mathbb{E}\left[\left\lvert f(Y)-f(x)\right\rvert \mathbf{1}_{A_{n,\delta}(x)^{c}}\right]\le\Omega_{\nu,*}(f;\delta)$.
On the far event, the weighted norm gives $\left\lvert f(Y)-f(x)\right\rvert \le\left\lvert f(Y)\right\rvert +\left\lvert f(x)\right\rvert \le\left\lVert f\right\rVert _{\infty,\nu}\left(w_{\nu}(Y)+w_{\nu}(x)\right)$,
hence 
\begin{align*}
\frac{1}{w_{\nu}(x)}\mathbb{E}\left[\left\lvert f(Y)-f(x)\right\rvert \mathbf{1}_{A_{n,\delta}(x)}\right] & \le\left\lVert f\right\rVert _{\infty,\nu}\left(\mathbb{P}(A_{n,\delta}(x))+\frac{\mathbb{E}[w_{\nu}(Y)\mathbf{1}_{A_{n,\delta}(x)}]}{w_{\nu}(x)}\right).
\end{align*}
By Lemma~\ref{lem:displacement-moments}, 
\begin{align*}
\mathbb{P}(A_{n,\delta}(x)) & \le\frac{\mathbb{E}[\left\lVert \Delta\right\rVert _{2}^{2}]}{\delta^{2}(1+\left\lVert x\right\rVert _{1})}\le\frac{1+d/3}{n\delta^{2}}.
\end{align*}
For the weighted expectation term, Cauchy--Schwarz and Lemma~\ref{lem:weighted-moments}
at the index $2\nu$ yield 
\begin{align*}
\frac{\mathbb{E}[w_{\nu}(Y)\mathbf{1}_{A_{n,\delta}(x)}]}{w_{\nu}(x)} & \le\frac{\mathbb{E}[w_{2\nu}(Y)]^{1/2}}{w_{\nu}(x)}\,\mathbb{P}(A_{n,\delta}(x))^{1/2}\\
 & \le\sqrt{A_{2\nu,d}}\,\mathbb{P}(A_{n,\delta}(x))^{1/2}\\
 & \le\frac{\sqrt{A_{2\nu,d}(1+d/3)}}{\sqrt{n}\,\delta}.
\end{align*}
Combining the near- and far-event estimates proves \eqref{eq:qualitative-weighted-estimate-main}.
To obtain convergence, fix $\delta>0$ and let $n\to\infty$, then
let $\delta\downarrow0$ and use \eqref{eq:intrinsic-modulus-zero-main}. 
\end{proof}
\begin{proof}[Proof of Corollary~\ref{cor:global-uniform-c0}]
The rate bound is immediate from Theorem~\ref{thm:weighted-holder-rate}
with $\nu=0$. For the qualitative convergence, it is enough to verify
that $\Omega_{0,*}(f;\delta)\to0$ as $\delta\downarrow0$, because
then Theorem~\ref{thm:weighted-qualitative} applies with $\nu=0$.

Fix $\varepsilon>0$. Since $f\in C_{0}(\mathbb{R}_{+}^{d})$, there
exists $R>0$ such that 
\[
\left\lvert f(z)\right\rvert <\frac{\varepsilon}{4}\qquad\text{whenever }\left\lVert z\right\rVert _{1}\ge\frac{R}{2}.
\]
Increase $R$, if necessary, so that 
\[
\sqrt{d}\sqrt{1+s}\le\frac{s}{2}\qquad(s\ge R).
\]
Set 
\[
K=\left\{ z\in\mathbb{R}_{+}^{d}:\left\lVert z\right\rVert _{1}\le R+\sqrt{d}\sqrt{1+R}\right\}.
\]
This set is compact, so $f$ is uniformly continuous on $K$. Choose
$\eta>0$ such that 
\[
\left\lvert f(y)-f(x)\right\rvert <\frac{\varepsilon}{2}\qquad(x,y\in K,\ \left\lVert y-x\right\rVert _{2}\le\eta).
\]
Now choose $0<\delta\le1$ so small that $\delta\sqrt{1+R}\le\eta$.

Let $x,y\in\mathbb{R}_{+}^{d}$ satisfy $\left\lVert y-x\right\rVert _{2}\le\delta\sqrt{1+\left\lVert x\right\rVert _{1}}$.
If $\left\lVert x\right\rVert _{1}\le R$, then $\left\lVert y-x\right\rVert _{2}\le\eta$
and 
\[
\left\lVert y\right\rVert _{1}\le\left\lVert x\right\rVert _{1}+\left\lVert y-x\right\rVert _{1}\le R+\sqrt{d}\left\lVert y-x\right\rVert _{2}\le R+\sqrt{d}\sqrt{1+R},
\]
so $x,y\in K$, and therefore $\left\lvert f(y)-f(x)\right\rvert <\varepsilon/2$.

If instead $\left\lVert x\right\rVert _{1}\ge R$, then, since $\delta\le1$,
\[
\left\lVert y\right\rVert _{1}\ge\left\lVert x\right\rVert _{1}-\left\lVert y-x\right\rVert _{1}\ge\left\lVert x\right\rVert _{1}-\sqrt{d}\left\lVert y-x\right\rVert _{2}\ge\left\lVert x\right\rVert _{1}-\sqrt{d}\sqrt{1+\left\lVert x\right\rVert _{1}}\ge\frac{\left\lVert x\right\rVert _{1}}{2}\ge\frac{R}{2}.
\]
Hence both $x$ and $y$ lie in the region where $\left\lvert f\right\rvert <\varepsilon/4$,
so again $\left\lvert f(y)-f(x)\right\rvert <\varepsilon/2$.

We have proved that $\Omega_{0,*}(f;\delta)\le\varepsilon/2$. Since
$\varepsilon>0$ was arbitrary, $\Omega_{0,*}(f;\delta)\to0$ as $\delta\downarrow0$.
Theorem~\ref{thm:weighted-qualitative} now gives the global uniform
convergence. Finally, because $\mathcal{K}_{n}^{(d)}1=1$, the same
argument applies to every bounded continuous $f$ such that $f-L\in C_{0}(\mathbb{R}_{+}^{d})$
for some constant $L\in\mathbb{R}$. 
\end{proof}
\begin{proof}[Proof of Corollary~\ref{cor:weighted-finite}]
Write $\mathcal{K}_{n}^{(d)}f=\sum_{m\in\mathbb{N}^{d}}a_{m,n}\,\varphi_{m,n}$,
where $\varphi_{m,n}(x)=\prod_{j=1}^{d}\tau_{m_{j},n}(x_{j})$. For
$N\in\mathbb{N}$, let $F_{N}=\{m\in\mathbb{N}^{d}:1\le m_{j}\le N\text{ for all }j=1,\dots,d\}$,
set $r_{n,N}=\sum_{m\notin F_{N}}a_{m,n}$, and define
\[
g_{n,N}=\sum_{m\in F_{N}}a_{m,n}\,\varphi_{m,n}+r_{n,N}\,\varphi_{\mathbf{1},n},
\qquad \mathbf{1}=(1,\dots,1)\in\mathbb{N}^{d}.
\]
Because $w_{\nu}(x)\ge1$ for all $x$ and $\tau_{m_{j},n}\le n$,
every kernel satisfies $\left\lVert \varphi_{m,n}\right\rVert _{\infty,\nu}\le\left\lVert \varphi_{m,n}\right\rVert _{L^{\infty}(\mathbb{R}_{+}^{d})}\le n^{d}$.
Hence
\begin{align*}
\left\lVert \mathcal{K}_{n}^{(d)}f-g_{n,N}\right\rVert _{\infty,\nu}
&\le\sum_{m\notin F_{N}}a_{m,n}\left(\left\lVert \varphi_{m,n}\right\rVert _{\infty,\nu}+\left\lVert \varphi_{\mathbf{1},n}\right\rVert _{\infty,\nu}\right)\\
&\le 2n^{d}r_{n,N}.
\end{align*}
Since $r_{n,N}\downarrow0$ as $N\to\infty$ for each fixed $n$, we may choose
$N(n)$ so large that the right-hand side is at most $1/n$, and then set
$g_{n}=g_{n,N(n)}$. It follows that
\[
\left\lVert g_{n}-f\right\rVert _{\infty,\nu}\le\left\lVert g_{n}-\mathcal{K}_{n}^{(d)}f\right\rVert _{\infty,\nu}+\left\lVert \mathcal{K}_{n}^{(d)}f-f\right\rVert _{\infty,\nu}\longrightarrow0.
\]
\end{proof}

\subsection{Probabilistic preliminaries and scale-rate proofs}
\begin{proof}[Proof of Proposition~\ref{prop:prob-representation}]
Fix $x\in\mathbb{R}_{+}^{d}$ and a measurable, locally integrable
$f$ such that $\mathbb{E}[|f(Y_{n,x})|]<\infty$. Since 
\[
\sum_{k\in\mathbb{N}_{0}^{d}}\mathbb{P}(N=k)\int_{[0,1]^{d}}\left|f\!\left(\frac{k+u}{n}\right)\right|\,\mathrm{d}u=\mathbb{E}[|f(Y_{n,x})|]<\infty,
\]
Fubini's theorem applies. Conditioning on the Poisson vector $N=(N_{1},\dots,N_{d})$
and then integrating over the independent uniforms gives 
\begin{align*}
\mathbb{E}[f(Y_{n,x})] & =\sum_{k\in\mathbb{N}_{0}^{d}}\mathbb{P}(N=k)\int_{[0,1]^{d}}f\!\left(\frac{k+u}{n}\right)\,\mathrm{d}u\\
 & =\sum_{k\in\mathbb{N}_{0}^{d}}\left(\prod_{j=1}^{d}e^{-nx_{j}}\frac{(nx_{j})^{k_{j}}}{k_{j}!}\right)\int_{[0,1]^{d}}f\!\left(\frac{k+u}{n}\right)\,\mathrm{d}u.
\end{align*}
Changing variables componentwise by $t_{j}=(k_{j}+u_{j})/n$ yields
$\mathrm{d}u=n^{d}\,\mathrm{d}t$ and carries $[0,1]^{d}$ onto the cell $\prod_{j=1}^{d}[k_{j}/n,(k_{j}+1)/n)$.
Hence 
\begin{align*}
\mathbb{E}[f(Y_{n,x})] & =n^{d}\sum_{k\in\mathbb{N}_{0}^{d}}\left(\prod_{j=1}^{d}e^{-nx_{j}}\frac{(nx_{j})^{k_{j}}}{k_{j}!}\right)\int_{\prod_{j=1}^{d}[k_{j}/n,(k_{j}+1)/n)}f(t)\,\mathrm{d}t.
\end{align*}
The same computation with $|f|$ shows that the series on the right-hand
side is absolutely convergent. This is exactly the cell expansion obtained
by iterating the one-dimensional definition \eqref{eq:Knj-def-main}, and
therefore proves \eqref{eq:prob-representation-main}. 
\end{proof}
\begin{proof}[Proof of Lemma~\ref{lem:displacement-moments}]
Fix $j\in\{1,\dots,d\}$. Since $Y_{n,x,j}=N_{j}/n+U_{j}/n$, with
$\mathbb{E}[N_{j}]=nx_{j}$, $\mathrm{Var}(N_{j})=nx_{j}$, $\mathbb{E}[U_{j}]=1/2$,
and $\mathrm{Var}(U_{j})=1/12$, we have $\mathbb{E}[Y_{n,x,j}-x_{j}]=1/(2n)$
and $\mathrm{Var}(Y_{n,x,j})=(\mathrm{Var}(N_{j})+\mathrm{Var}(U_{j}))/n^{2}=x_{j}/n+1/(12n^{2})$.
Therefore, $\mathbb{E}[(Y_{n,x,j}-x_{j})^{2}]=\mathrm{Var}(Y_{n,x,j})+\left(\mathbb{E}[Y_{n,x,j}-x_{j}]\right)^{2}=x_{j}/n+1/(3n^{2})$.
Summing over $j$ gives \eqref{eq:second-moment-displacement-main}.
If $0<r\le2$, then $t\mapsto t^{r/2}$ is concave on $[0,\infty)$,
so Jensen's inequality implies 
\[
\mathbb{E}\left[\left\lVert \Delta_{n,x}\right\rVert _{2}^{r}\right]\le\mathbb{E}\left[\left\lVert \Delta_{n,x}\right\rVert _{2}^{2}\right]^{r/2}=\left(\frac{\left\lVert x\right\rVert _{1}}{n}+\frac{d}{3n^{2}}\right)^{r/2}\le(1+d/3)^{r/2}\left(\frac{1+\left\lVert x\right\rVert _{1}}{n}\right)^{r/2},
\]
which proves \eqref{eq:r-moment-displacement-main}. The tail bound
\eqref{eq:tail-displacement-main} is then just Chebyshev's inequality. 
\end{proof}
\begin{proof}[Proof of Lemma~\ref{lem:weighted-moments}]
Fix $x\in\mathbb{R}_{+}^{d}$ and write $S_{x}=\left\lVert x\right\rVert _{1}$.
Let $\Sigma_{n}=N_{1}+\cdots+N_{d}$ and $V=U_{1}+\cdots+U_{d}$.
Then $\Sigma_{n}\sim\mathrm{Poisson}(nS_{x})$, $0\le V\le d$, and
$\left\lVert Y_{n,x}\right\rVert _{1}=(\Sigma_{n}+V)/n$. If $\nu=0$,
then $w_{0}=1$ and \eqref{eq:weighted-moment-bound-main} holds with
$A_{0,d}=1$. Assume henceforth that $\nu>0$, and choose an integer
$m\ge\nu$, for example $m=\lceil\nu\rceil$. By monotonicity of $L^{q}$
norms, $\mathbb{E}[(1+\left\lVert Y_{n,x}\right\rVert _{1})^{\nu}]\le\mathbb{E}[(1+\left\lVert Y_{n,x}\right\rVert _{1})^{m}]^{\nu/m}$,
so it suffices to control the $m$th moment. Since $0\le V\le d$,
we have $1+\left\lVert Y_{n,x}\right\rVert _{1}\le1+\Sigma_{n}/n+d$,
and therefore $(1+\left\lVert Y_{n,x}\right\rVert _{1})^{m}\le C_{m,d}\left(1+\left(\frac{\Sigma_{n}}{n}\right)^{m}\right)$
for a constant $C_{m,d}>0$. The $m$th moment of a Poisson random
variable with mean $\lambda$ is the Touchard polynomial $T_{m}(\lambda)$,
a polynomial of degree $m$ with nonnegative coefficients; see, for
example, \cite[Ch.~3]{Charalambides2002}. In particular, there
exists $B_{m}>0$ such that $\mathbb{E}[\Sigma_{n}^{m}]\le B_{m}(1+nS_{x})^{m}$.
Therefore $\mathbb{E}\left[\left(\frac{\Sigma_{n}}{n}\right)^{m}\right]\le B_{m}(1+nS_{x})^{m}/n^{m}\le B_{m}(1+S_{x})^{m}$,
because $1+nS_{x}\le n(1+S_{x})$ for $n\ge1$. Combining the last
estimates yields $\mathbb{E}[(1+\left\lVert Y_{n,x}\right\rVert _{1})^{m}]\le C'_{m,d}(1+S_{x})^{m}$,
and hence $\mathbb{E}[w_{\nu}(Y_{n,x})]=\mathbb{E}[(1+\left\lVert Y_{n,x}\right\rVert _{1})^{\nu}]\le A_{\nu,d}(1+S_{x})^{\nu}=A_{\nu,d}w_{\nu}(x)$.
This proves \eqref{eq:weighted-moment-bound-main}. The operator bound
\eqref{eq:weighted-operator-bound-main} follows from Proposition~\ref{prop:prob-representation},
because $\mathbb{E}[|f(Y_{n,x})|]\le\left\lVert f\right\rVert _{\infty,\nu}\,\mathbb{E}[w_{\nu}(Y_{n,x})]<\infty$:
$\left\lvert (\mathcal{K}_{n}^{(d)}f)(x)\right\rvert \le\mathbb{E}[\left\lvert f(Y_{n,x})\right\rvert ]\le\left\lVert f\right\rVert _{\infty,\nu}\,\mathbb{E}[w_{\nu}(Y_{n,x})]\le A_{\nu,d}\,\left\lVert f\right\rVert _{\infty,\nu}\,w_{\nu}(x)$,
and now take the supremum over $x$. 
\end{proof}
\begin{proof}[Proof of Theorem~\ref{thm:compact-modulus}]
Fix $x\in Q_{M}$ and write $Y=Y_{n,x}$. By Proposition~\ref{prop:prob-representation},
$(\mathcal{K}_{n}^{(d)}f)(x)-f(x)=\mathbb{E}[f(Y)-f(x)]$. Split the
expectation according to the event $A=\{\left\lVert Y-x\right\rVert _{2}\le r\}$.
On $A$, the point $Y$ lies in $Q_{M+r}$, so $\left\lvert f(Y)-f(x)\right\rvert \le\omega_{Q_{M+r}}(f;r)$.
On $A^{c}$, we use the trivial bound $\left\lvert f(Y)-f(x)\right\rvert \le2\left\lVert f\right\rVert _{L^{\infty}(\mathbb{R}_{+}^{d})}$.
Therefore, 
\[
\left\lvert (\mathcal{K}_{n}^{(d)}f)(x)-f(x)\right\rvert \le\omega_{Q_{M+r}}(f;r)+2\left\lVert f\right\rVert _{L^{\infty}(\mathbb{R}_{+}^{d})}\,\mathbb{P}(A^{c}).
\]
By Lemma~\ref{lem:displacement-moments}, 
\[
\mathbb{P}(A^{c})=\mathbb{P}(\left\lVert Y-x\right\rVert _{2}>r)\le\frac{1}{r^{2}}\left(\frac{\left\lVert x\right\rVert _{1}}{n}+\frac{d}{3n^{2}}\right)\le\frac{1}{r^{2}}\left(\frac{dM}{n}+\frac{d}{3n^{2}}\right),
\]
because $x\in Q_{M}$ implies $\left\lVert x\right\rVert _{1}\le dM$.
Taking the supremum over $x\in Q_{M}$ proves \eqref{eq:compact-modulus-estimate-main}. 
\end{proof}
\begin{proof}[Proof of Corollary~\ref{cor:compact-uniform}]
Fix $M>0$ and $\varepsilon>0$. Since $f$ is continuous on the
compact cube $Q_{M+1}$, it is uniformly continuous there. Choose
$0<r\le1$ such that $\omega_{Q_{M+1}}(f;r)<\varepsilon/2$. Since
$Q_{M+r}\subseteq Q_{M+1}$, the first term in \eqref{eq:compact-modulus-estimate-main}
is at most $\varepsilon/2$. Keeping this $r$ fixed and letting $n\to\infty$,
the second term in \eqref{eq:compact-modulus-estimate-main} tends
to $0$. Hence $\left\lVert \mathcal{K}_{n}^{(d)}f-f\right\rVert _{L^{\infty}(Q_{M})}<\varepsilon$
for all sufficiently large $n$. 
\end{proof}
\begin{proof}[Proof of Theorem~\ref{thm:compact-holder}]
Fix $x\in Q_{M}$ and write $Y=Y_{n,x}$. Split according to the
event $A=\{\left\lVert Y-x\right\rVert _{2}\le1\}$. On $A$, the
point $Y$ lies in $Q_{M+1}$, so the local Hölder condition \eqref{eq:local-holder-compact-main}
yields $\left\lvert f(Y)-f(x)\right\rvert \le H\left\lVert Y-x\right\rVert _{2}^{\alpha}$.
On $A^{c}$ we use the crude estimate $\left\lvert f(Y)-f(x)\right\rvert \le2\left\lVert f\right\rVert _{L^{\infty}(\mathbb{R}_{+}^{d})}$.
Therefore, 
\[
\left\lvert (\mathcal{K}_{n}^{(d)}f)(x)-f(x)\right\rvert \le H\mathbb{E}[\left\lVert Y-x\right\rVert _{2}^{\alpha}]+2\left\lVert f\right\rVert _{L^{\infty}(\mathbb{R}_{+}^{d})}\mathbb{P}(A^{c}).
\]
By Lemma~\ref{lem:displacement-moments}, 
\[
\mathbb{E}[\left\lVert Y-x\right\rVert _{2}^{\alpha}]\le C_{\alpha,d}\left(\frac{1+\left\lVert x\right\rVert _{1}}{n}\right)^{\alpha/2}\le C_{\alpha,d}\left(\frac{1+dM}{n}\right)^{\alpha/2},
\]
and 
\[
\mathbb{P}(A^{c})\le\frac{\left\lVert x\right\rVert _{1}}{n}+\frac{d}{3n^{2}}\le\frac{dM}{n}+\frac{d}{3n^{2}}.
\]
Taking the supremum over $x\in Q_{M}$ gives \eqref{eq:compact-holder-rate-main}. 
\end{proof}
\begin{proof}[Proof of Theorem~\ref{thm:weighted-holder-rate}]
Fix $x\in\mathbb{R}_{+}^{d}$ and write $Y=Y_{n,x}$. Since $f\in L_{\infty,\nu}(\mathbb{R}_{+}^{d})$
and Lemma~\ref{lem:weighted-moments} gives $\mathbb{E}[w_{\nu}(Y)]<\infty$,
we have $\mathbb{E}[|f(Y)|]<\infty$. Hence Proposition~\ref{prop:prob-representation}
applies, and $(\mathcal{K}_{n}^{(d)}f)(x)-f(x)=\mathbb{E}[f(Y)-f(x)]$.
Since $[f]_{\nu,\alpha,*}<\infty$, we may estimate pointwise by $\left\lvert f(Y)-f(x)\right\rvert \le[f]_{\nu,\alpha,*}\,w_{\nu}(x)\left(\frac{\left\lVert Y-x\right\rVert _{2}}{\sqrt{1+\left\lVert x\right\rVert _{1}}}\right)^{\alpha}$.
Taking expectations and applying Lemma~\ref{lem:displacement-moments}
yields 
\[
\frac{\left\lvert (\mathcal{K}_{n}^{(d)}f)(x)-f(x)\right\rvert }{w_{\nu}(x)}\le[f]_{\nu,\alpha,*}\frac{\mathbb{E}[\left\lVert Y-x\right\rVert _{2}^{\alpha}]}{(1+\left\lVert x\right\rVert _{1})^{\alpha/2}}\le C_{\alpha,d}[f]_{\nu,\alpha,*}\,n^{-\alpha/2}.
\]
Taking the supremum over $x$ proves \eqref{eq:weighted-holder-rate-main}. 
\end{proof}
\begin{proof}[Proof of Theorem~\ref{thm:weighted-lp-rate}]
Fix $x\in\mathbb{R}_{+}^{d}$ and write $Y=Y_{n,x}$. Since $H\in L_{p,\eta}(\mathbb{R}_{+}^{d})$,
there exists $x_{0}\in\mathbb{R}_{+}^{d}$ such that $H(x_{0})<\infty$.
Applying \eqref{eq:weighted-lp-envelope-main} with $x=x_{0}$ gives
\[
|f(y)|\le|f(x_{0})|+H(x_{0})\left(\frac{\left\lVert y-x_{0}\right\rVert _{2}}{\sqrt{1+\left\lVert x_{0}\right\rVert _{1}}}\right)^{\alpha}\qquad(y\in\mathbb{R}_{+}^{d}),
\]
so $f$ has at most polynomial growth of order $\alpha$ and, in particular,
is locally integrable. Since $0<\alpha\le1$,
$\left\lVert Y\right\rVert _{2}^{\alpha}\le\left\lVert x\right\rVert _{2}^{\alpha}+\left\lVert Y-x\right\rVert _{2}^{\alpha}$,
and Lemma~\ref{lem:displacement-moments} implies $\mathbb{E}[\left\lVert Y-x\right\rVert _{2}^{\alpha}]<\infty$.
Therefore $\mathbb{E}[\left\lVert Y\right\rVert _{2}^{\alpha}]<\infty$,
and hence $\mathbb{E}[|f(Y)|]<\infty$. Hence Proposition~\ref{prop:prob-representation}
applies and $(\mathcal{K}_{n}^{(d)}f)(x)-f(x)=\mathbb{E}[f(Y)-f(x)]$.
Taking absolute values and using \eqref{eq:weighted-lp-envelope-main}
gives $\left\lvert (\mathcal{K}_{n}^{(d)}f)(x)-f(x)\right\rvert \le H(x)\,\mathbb{E}\left[\left(\frac{\left\lVert Y-x\right\rVert _{2}}{\sqrt{1+\left\lVert x\right\rVert _{1}}}\right)^{\alpha}\right]$.
Lemma~\ref{lem:displacement-moments} yields 
\[
\mathbb{E}\left[\left(\frac{\left\lVert Y-x\right\rVert _{2}}{\sqrt{1+\left\lVert x\right\rVert _{1}}}\right)^{\alpha}\right]\le C_{\alpha,d}\,n^{-\alpha/2},
\]
so pointwise on $\mathbb{R}_{+}^{d}$, 
\[
\left\lvert (\mathcal{K}_{n}^{(d)}f)(x)-f(x)\right\rvert \le C_{\alpha,d}\,n^{-\alpha/2}H(x).
\]
Divide by $w_{\eta}(x)^{1/p}$, raise to the $p$th power, integrate
over $\mathbb{R}_{+}^{d}$, and take the $p$th root. This gives \eqref{eq:weighted-lp-rate-main}. 
\end{proof}
\begin{proof}[Proof of Corollary~\ref{cor:weighted-sup-to-lp}]
By Theorem~\ref{thm:weighted-holder-rate}, 
\[
\left\lvert (\mathcal{K}_{n}^{(d)}f)(x)-f(x)\right\rvert \le C_{\alpha,d}[f]_{\nu,\alpha,*}\,n^{-\alpha/2}w_{\nu}(x).
\]
Therefore, 
\begin{align*}
\left\lVert \mathcal{K}_{n}^{(d)}f-f\right\rVert _{p,\eta}^{p} & \le C_{\alpha,d}^{p}[f]_{\nu,\alpha,*}^{p}n^{-\alpha p/2}\int_{\mathbb{R}_{+}^{d}}\frac{w_{\nu}(x)^{p}}{w_{\eta}(x)}\,\mathrm{d}x\\
 & =C_{\alpha,d}^{p}[f]_{\nu,\alpha,*}^{p}n^{-\alpha p/2}\int_{\mathbb{R}_{+}^{d}}(1+\left\lVert x\right\rVert _{1})^{\nu p-\eta}\,\mathrm{d}x.
\end{align*}
The final integral is finite because $\eta-\nu p>d$. Taking $p$th
roots proves \eqref{eq:weighted-sup-to-lp-main}. 
\end{proof}

\subsection{Proofs for component-count approximation}
\begin{proof}[Proof of Proposition~\ref{prop:compact-tailfree}]
The only new point is the truncation estimate. Since $\mathcal{K}_{n}^{(d)}f-\widetilde{g}_{n,N}^{(M)}=\sum_{m\notin F_{N}}a_{m,n}(\varphi_{m,n}-\varphi_{\ell_{N},n})$,
it suffices to bound the kernels on $Q_{M}$. Fix $m\notin F_{N}$.
Then some coordinate $j$ satisfies $m_{j}\ge N+1$, so $m_{j}-1\ge N\ge nM$.
The one-dimensional Erlang kernel $x\mapsto\tau_{m_{j},n}(x)$ has
mode at $(m_{j}-1)/n$, which lies to the right of $M$. Hence, for
$0\le x\le M$, $\tau_{m_{j},n}(x)\le\tau_{m_{j},n}(M)=ne^{-nM}\frac{(nM)^{m_{j}-1}}{(m_{j}-1)!}$.
Moreover, because $m_{j}-1\ge N\ge nM$, the sequence $r\mapsto(nM)^{r}/r!$
is decreasing for $r\ge N$, so $\tau_{m_{j},n}(x)\le ne^{-nM}\frac{(nM)^{N}}{N!}$
for $0\le x\le M$. Using also the crude bound $\tau_{m_{i},n}\le n$
from Lemma~\ref{lem:kernel-bound} for the remaining coordinates,
we obtain $\left\lVert \varphi_{m,n}\right\rVert _{L^{\infty}(Q_{M})}\le n^{d}e^{-nM}\frac{(nM)^{N}}{N!}$.
Exactly the same argument applies to the index $\ell_{N}$,
since every coordinate of $\ell_{N}$ equals $N+1$. Therefore, 
\begin{align*}
\left\lVert \mathcal{K}_{n}^{(d)}f-\widetilde{g}_{n,N}^{(M)}\right\rVert _{L^{\infty}(Q_{M})} & \le\sum_{m\notin F_{N}}a_{m,n}\left(\left\lVert \varphi_{m,n}\right\rVert _{L^{\infty}(Q_{M})}+\left\lVert \varphi_{\ell_{N},n}\right\rVert _{L^{\infty}(Q_{M})}\right)\\
 & \le2n^{d}e^{-nM}\frac{(nM)^{N}}{N!}\sum_{m\notin F_{N}}a_{m,n}\\
 & \le2n^{d}e^{-nM}\frac{(nM)^{N}}{N!},
\end{align*}
which proves \eqref{eq:compact-tailfree-trunc-main}. The estimate
\eqref{eq:compact-tailfree-modulus-main} follows by combining this
bound with Theorem~\ref{thm:compact-modulus}. 
\end{proof}
\begin{proof}[Proof of Corollary~\ref{cor:compact-component-rate}]
For each $n\in\mathbb{N}$, choose $N_{n}=\lceil n(M+1)\rceil$ and
define $g_{n}=\widetilde{g}_{n,N_{n}}^{(M)}$. By Proposition~\ref{prop:compact-tailfree}
and Theorem~\ref{thm:compact-holder}, 
\begin{align*}
\left\lVert g_{n}-f\right\rVert _{L^{\infty}(Q_{M})} & \le HC_{\alpha,d}\left(\frac{1+dM}{n}\right)^{\alpha/2}+2\left\lVert f\right\rVert _{L^{\infty}(\mathbb{R}_{+}^{d})}\left(\frac{dM}{n}+\frac{d}{3n^{2}}\right)\\
 & \qquad+2n^{d}e^{-nM}\frac{(nM)^{N_{n}}}{N_{n}!}.
\end{align*}
The first two terms are $O(n^{-\alpha/2})$. Since $N_{n}/n\to M+1>M$,
Stirling's formula gives $e^{-nM}\frac{(nM)^{N_{n}}}{N_{n}!}=O(e^{-c_{M}n})$
for some $c_{M}>0$. Hence there exist constants $C_{1}>0$ and $n_{0}\in\mathbb{N}$
such that $\left\lVert g_{n}-f\right\rVert _{L^{\infty}(Q_{M})}\le C_{1}n^{-\alpha/2}$
for $n\ge n_{0}$. Also, $g_{n}$ has at most $N_{n}^{d}+1\le B_{M}n^{d}$
components for some constant $B_{M}>0$. Now fix an integer $K\ge B_{M}(2n_{0})^{d}$
and set $n(K)=\left\lfloor \left(\frac{K}{B_{M}}\right)^{1/d}\right\rfloor $.
Define $g_{K}=g_{n(K)}$. Then $g_{K}$ has at most $K$ components,
and because $\left(K/B_{M}\right)^{1/d}\ge2n_{0}$, we have $n(K)\ge n_{0}$.
Also, since $\lfloor t\rfloor\ge t/2$ for all $t\ge1$, $n(K)\ge\frac{1}{2}\left(\frac{K}{B_{M}}\right)^{1/d}$.
Therefore $\left\lVert g_{K}-f\right\rVert _{L^{\infty}(Q_{M})}\le C_{1}n(K)^{-\alpha/2}\le C_{1}2^{\alpha/2}B_{M}^{\alpha/(2d)}K^{-\alpha/(2d)}$.
Thus the conclusion holds with $C=C_{1}2^{\alpha/2}B_{M}^{\alpha/(2d)}$
and $K_{0}=\left\lceil B_{M}(2n_{0})^{d}\right\rceil $. 
\end{proof}
\begin{proof}[Proof of Lemma~\ref{lem:erlang-shape-sup}]
For $m=1$, we have $\left\lVert \tau_{1,n}\right\rVert _{L^{\infty}(\mathbb{R}_{+})}=n$.
If $m\ge2$, then $\tau_{m,n}$ attains its maximum at the mode $x=(m-1)/n$,
so $\left\lVert \tau_{m,n}\right\rVert _{L^{\infty}(\mathbb{R}_{+})}=ne^{-(m-1)}\frac{(m-1)^{m-1}}{(m-1)!}$.
Stirling's lower bound $(m-1)!\ge\sqrt{2\pi(m-1)}\,((m-1)/e)^{m-1}$
therefore yields $\left\lVert \tau_{m,n}\right\rVert _{L^{\infty}(\mathbb{R}_{+})}\le\frac{n}{\sqrt{2\pi(m-1)}}\le nm^{-1/2}$.
For the consequence, note first that $w_{\nu}(x)\ge1$, so $\left\lVert \varphi_{m,n}\right\rVert _{\infty,\nu}\le\left\lVert \varphi_{m,n}\right\rVert _{L^{\infty}(\mathbb{R}_{+}^{d})}=\prod_{i=1}^{d}\left\lVert \tau_{m_{i},n}\right\rVert _{L^{\infty}(\mathbb{R}_{+})}$.
If some coordinate $m_{j}\ge N+1$, then the one-dimensional estimate
above gives $\left\lVert \tau_{m_{j},n}\right\rVert _{L^{\infty}}\le n(N+1)^{-1/2}$,
while the remaining factors are at most $n$. Hence 
\[
\left\lVert \varphi_{m,n}\right\rVert _{\infty,\nu}\le n^{d}(N+1)^{-1/2}.
\]
\end{proof}
\begin{proof}[Proof of Proposition~\ref{prop:weighted-tailfree}]
As before, $\mathcal{K}_{n}^{(d)}f-\widetilde{g}_{n,N}^{(\nu)}=\sum_{m\notin F_{N}}a_{m,n}(\varphi_{m,n}-\varphi_{\ell_{N},n})$.
If $m\notin F_{N}$, then some coordinate of $m$ is at least $N+1$,
so Lemma~\ref{lem:erlang-shape-sup} gives $\left\lVert \varphi_{m,n}\right\rVert _{\infty,\nu}\le n^{d}(N+1)^{-1/2}$.
The same bound applies to the kernel $\varphi_{\ell_{N},n}$.
Therefore, 
\begin{align*}
\left\lVert \mathcal{K}_{n}^{(d)}f-\widetilde{g}_{n,N}^{(\nu)}\right\rVert _{\infty,\nu} & \le\sum_{m\notin F_{N}}a_{m,n}\left(\left\lVert \varphi_{m,n}\right\rVert _{\infty,\nu}+\left\lVert \varphi_{\ell_{N},n}\right\rVert _{\infty,\nu}\right)\\
 & \le2n^{d}(N+1)^{-1/2}\sum_{m\notin F_{N}}a_{m,n}\\
 & \le2n^{d}(N+1)^{-1/2},
\end{align*}
which proves \eqref{eq:weighted-tailfree-trunc-main}. 
\end{proof}
\begin{proof}[Proof of Corollary~\ref{cor:weighted-component-rate}]
For each $n\in\mathbb{N}$, choose $N_{n}=\lceil n^{2d+\alpha}\rceil$
and define $g_{n}=\widetilde{g}_{n,N_{n}}^{(\nu)}$. By Proposition~\ref{prop:weighted-tailfree}
and Theorem~\ref{thm:weighted-holder-rate}, 
\[
\left\lVert g_{n}-f\right\rVert _{\infty,\nu}\le C_{\alpha,d}[f]_{\nu,\alpha,*}\,n^{-\alpha/2}+2n^{d}(N_{n}+1)^{-1/2}\le C_{1}n^{-\alpha/2}
\]
for some constant $C_{1}>0$ and all sufficiently large $n$. Choose
$n_{0}\in\mathbb{N}$ so that this bound holds for every $n\ge n_{0}$.
Moreover, $g_{n}$ has at most $N_{n}^{d}+1\le Bn^{d(2d+\alpha)}$
components for some constant $B>0$. Now fix an integer $K\ge B(2n_{0})^{d(2d+\alpha)}$
and set $n(K)=\left\lfloor \left(\frac{K}{B}\right)^{1/[d(2d+\alpha)]}\right\rfloor $.
Define $g_{K}=g_{n(K)}$. Then $g_{K}$ has at most $K$ components,
and because $\left(\frac{K}{B}\right)^{1/[d(2d+\alpha)]}\ge2n_{0}$,
we have $n(K)\ge n_{0}$. Also, since $\lfloor t\rfloor\ge t/2$ for
all $t\ge1$, $n(K)\ge\frac{1}{2}\left(\frac{K}{B}\right)^{1/[d(2d+\alpha)]}$.
Therefore $\left\lVert g_{K}-f\right\rVert _{\infty,\nu}\le C_{1}n(K)^{-\alpha/2}\le C_{1}2^{\alpha/2}B^{\alpha/[2d(2d+\alpha)]}K^{-\alpha/[2d(2d+\alpha)]}$.
Thus the conclusion holds with $C=C_{1}2^{\alpha/2}B^{\alpha/[2d(2d+\alpha)]}$
and $K_{0}=\left\lceil B(2n_{0})^{d(2d+\alpha)}\right\rceil $. 
\end{proof}
\begin{proof}[Proof of Lemma~\ref{lem:erlang-shape-lp}]
The inequality $\|u\|_{L^{p}}\le\|u\|_{L^{1}}^{1/p}\|u\|_{L^{\infty}}^{1-1/p}$
for $u\in L^{1}\cap L^{\infty}$, together with $\|\tau_{m,n}\|_{L^{1}(\mathbb{R}_{+})}=1$
and Lemma~\ref{lem:erlang-shape-sup}, gives $\left\lVert \tau_{m,n}\right\rVert _{L^{p}(\mathbb{R}_{+})}\le\left\lVert \tau_{m,n}\right\rVert _{L^{\infty}(\mathbb{R}_{+})}^{1-1/p}\le\left(nm^{-1/2}\right)^{1-1/p}=n^{1-1/p}m^{-(1-1/p)/2}$.
For the consequence, note that $w_{\eta}(x)\ge1$, so $\left\lVert \varphi_{m,n}\right\rVert _{p,\eta}\le\left\lVert \varphi_{m,n}\right\rVert _{L^{p}(\mathbb{R}_{+}^{d})}=\prod_{i=1}^{d}\left\lVert \tau_{m_{i},n}\right\rVert _{L^{p}(\mathbb{R}_{+})}$.
If some coordinate $m_{j}\ge N+1$, then the one-dimensional estimate
above gives $\|\tau_{m_{j},n}\|_{L^{p}}\le n^{1-1/p}(N+1)^{-(1-1/p)/2}$,
while the remaining factors are at most $n^{1-1/p}$. Multiplying
the factors yields the stated bound. 
\end{proof}
\begin{proof}[Proof of Proposition~\ref{prop:weighted-lp-tailfree}]
As before, $\mathcal{K}_{n}^{(d)}f-\widetilde{g}_{n,N}^{(p,\eta)}=\sum_{m\notin F_{N}}a_{m,n}(\varphi_{m,n}-\varphi_{\ell_{N},n})$.
If $m\notin F_{N}$, then some coordinate of $m$ is at least $N+1$,
so Lemma~\ref{lem:erlang-shape-lp} gives $\|\varphi_{m,n}\|_{p,\eta}\le n^{d(1-1/p)}(N+1)^{-(1-1/p)/2}$.
The same bound applies to the kernel $\varphi_{\ell_{N},n}$.
Therefore, 
\begin{align*}
\left\lVert \mathcal{K}_{n}^{(d)}f-\widetilde{g}_{n,N}^{(p,\eta)}\right\rVert _{p,\eta} & \le\sum_{m\notin F_{N}}a_{m,n}\left(\left\lVert \varphi_{m,n}\right\rVert _{p,\eta}+\left\lVert \varphi_{\ell_{N},n}\right\rVert _{p,\eta}\right)\\
 & \le2n^{d(1-1/p)}(N+1)^{-(1-1/p)/2}\sum_{m\notin F_{N}}a_{m,n}\\
 & \le2n^{d(1-1/p)}(N+1)^{-(1-1/p)/2},
\end{align*}
which proves \eqref{eq:weighted-lp-tailfree-trunc-main}. 
\end{proof}
\begin{proof}[Proof of Corollary~\ref{cor:weighted-lp-component-rate}]
Set $\gamma_{p,\alpha}=2d+\frac{\alpha p}{p-1}$. For each $n\in\mathbb{N}$,
choose $N_{n}=\lceil n^{\gamma_{p,\alpha}}\rceil$ and define $g_{n}=\widetilde{g}_{n,N_{n}}^{(p,\eta)}$.
By Proposition~\ref{prop:weighted-lp-tailfree} and Theorem~\ref{thm:weighted-lp-rate},
\begin{align*}
\left\lVert g_{n}-f\right\rVert _{p,\eta} & \le C_{\alpha,d}\,n^{-\alpha/2}\,\left\lVert H\right\rVert _{p,\eta}+2n^{d(1-1/p)}(N_{n}+1)^{-(1-1/p)/2}\\
 & \le C_{\alpha,d}\,n^{-\alpha/2}\,\left\lVert H\right\rVert _{p,\eta}+2n^{d(1-1/p)-\gamma_{p,\alpha}(1-1/p)/2}\\
 & =C_{\alpha,d}\,n^{-\alpha/2}\,\left\lVert H\right\rVert _{p,\eta}+2n^{-\alpha/2}\\
 & \le C_{1}n^{-\alpha/2}
\end{align*}
for some constant $C_{1}>0$ and all $n\in\mathbb{N}$. Moreover,
$g_{n}$ has at most $N_{n}^{d}+1\le Bn^{d\gamma_{p,\alpha}}$ components
for some constant $B>0$. Choose $n_{0}\in\mathbb{N}$ so that this
bound holds for every $n\ge n_{0}$, fix an integer $K\ge B(2n_{0})^{d\gamma_{p,\alpha}}$,
and set $n(K)=\left\lfloor \left(\frac{K}{B}\right)^{1/(d\gamma_{p,\alpha})}\right\rfloor $.
Define $g_{K}=g_{n(K)}$. Then $g_{K}$ has at most $K$ components,
and since $\left(\frac{K}{B}\right)^{1/(d\gamma_{p,\alpha})}\ge2n_{0}$,
we have $n(K)\ge n_{0}$ and $n(K)\ge\frac{1}{2}\left(\frac{K}{B}\right)^{1/(d\gamma_{p,\alpha})}$.
Therefore, 
\[
\left\lVert g_{K}-f\right\rVert _{p,\eta}\le C_{1}n(K)^{-\alpha/2}\le C_{1}2^{\alpha/2}B^{\alpha/(2d\gamma_{p,\alpha})}K^{-\alpha/(2d\gamma_{p,\alpha})}.
\]
Thus the conclusion holds with 
\[
C=C_{1}2^{\alpha/2}B^{\alpha/(2d\gamma_{p,\alpha})}\quad\text{and}\quad K_{0}=\left\lceil B(2n_{0})^{d\gamma_{p,\alpha}}\right\rceil .
\]
\end{proof}

\section*{Acknowledgements}

The author acknowledges support from the Australian Research Council
through Discovery Projects DP250100860 and DP230100905. AI was used in the proofreading of the final manuscript.

\section*{Competing interests}

The author declares no competing interests. 

\bigskip{}
 \textbf{Address for correspondence:}\\
 Hien Duy Nguyen\\
 School of Computing, Engineering and Mathematical Sciences, La Trobe
University\\
 Bundoora, VIC 3086, Australia\\
 Institute of Mathematics for Industry, Kyushu University\\
 Fukuoka 819-0395, Japan\\
 \texttt{H.Nguyen5@latrobe.edu.au} 

\begin{thebibliography}{10}
\bibitem{AltomareCampiti1994} Altomare, F., \& Campiti, M. (1994).
\emph{Korovkin-type approximation theory and its applications}. Berlin:
Walter de Gruyter.

\bibitem{AltomareCappellettiLeonessa2013} Altomare, F., Cappelletti
Montano, M., \& Leonessa, V. (2013). On a generalization of Szász--Mirakjan--Kantorovich
operators. \emph{Results in Mathematics}, \textbf{63}, 837--863.

\bibitem{AsmussenAlbrecher2010} Asmussen, S., \& Albrecher, H. (2010).
\emph{Ruin probabilities} (2nd ed.). Singapore: World Scientific.
\url{https://doi.org/10.1142/7431}

\bibitem{BagnatoPunzo2013} Bagnato, L., \& Punzo, A. (2013). Finite
mixtures of unimodal beta and gamma densities and the bumps algorithm.
\emph{Computational Statistics}, \textbf{28}(4), 1571--1597.

\bibitem{BarbosuEtAl2010} Bărbosu, D., Pop, O. T., \& Miclăuş, D.
(2010). The Kantorovich form of some extensions for the Szász--Mirakjan
operators. \emph{Revue d'Analyse Numérique et de Théorie de l'Approximation},
\textbf{39}(1), 8--20.

\bibitem{BochkinaRousseau2017} Bochkina, N., \& Rousseau, J. (2017).
Adaptive density estimation based on a mixture of gammas. \emph{Electronic
Journal of Statistics}, \textbf{11}(1), 916--962.

\bibitem{Boland2007} Boland, P. J. (2007). \emph{Statistical
and probabilistic methods in actuarial science}. Boca Raton, FL:
Chapman \& Hall/CRC.

\bibitem{Butzer1954} Butzer, P. L. (1954). On the extensions of Bernstein
polynomials to the infinite interval. \emph{Proceedings of the American
Mathematical Society}, \textbf{5}(4), 547--553.

\bibitem{Charalambides2002} Charalambides, Ch. A. (2002).
\emph{Enumerative combinatorics}. Boca Raton, FL: Chapman \& Hall/CRC.

\bibitem{DitzianTotik1987} Ditzian, Z., \& Totik, V. (1987). \emph{Moduli
of smoothness}. New York: Springer.

\bibitem{GuiHuangLin2021} Gui, W., Huang, R., \& Lin, X. S. (2021).
Fitting multivariate Erlang mixtures to data: A roughness penalty
approach. \emph{Journal of Computational and Applied Mathematics},
\textbf{386}, Article 113216.

\bibitem{KaasGoovaertsDhaeneDenuit2008} Kaas, R., Goovaerts, M.,
Dhaene, J., \& Denuit, M. (2008). \emph{Modern actuarial risk theory
using R} (2nd ed.). Berlin: Springer.

\bibitem{KimKottas2022} Kim, H., \& Kottas, A. (2022). Erlang mixture
modeling for Poisson process intensities. \emph{Statistics and Computing},
\textbf{32}, Article 3.

\bibitem{KlugmanPanjerWillmot2012} Klugman, S. A., Panjer, H. H.,
\& Willmot, G. E. (2012). \emph{Loss models: From data to decisions}
(4th ed.). Hoboken, NJ: Wiley.

\bibitem{LeeLin2010} Lee, S. C. K., \& Lin, X. S. (2010). Modeling
and evaluating insurance losses via mixtures of Erlang distributions.
\emph{North American Actuarial Journal}, \textbf{14}(1), 107--130.

\bibitem{LeeLin2012} Lee, S. C. K., \& Lin, X. S. (2012). Modeling
dependent risks with multivariate Erlang mixtures. \emph{ASTIN Bulletin:
The Journal of the IAA}, \textbf{42}(1), 153--180.

\bibitem{MahmudovKara2024} Mahmudov, N. I., \& Kara, M. (2024). New
Kantorovich-type Szász--Mirakjan operators. \emph{Bulletin of the
Iranian Mathematical Society}, \textbf{50}(5), Article 75.

\bibitem{Mirakyan1941} Mirakyan, G. M. (1941). Approximation des
fonctions continues au moyen de polynômes de la forme $e^{-nx}\sum_{k=0}^{\infty}f\!\left(\frac{k}{n}\right)\frac{(nx)^{k}}{k!}$.
\emph{Doklady Akademii Nauk SSSR}, \textbf{31}, 201--205.

\bibitem{NguyenEtAl2020} Nguyen, T. T., Nguyen, H. D., Chamroukhi,
F., \& McLachlan, G. J. (2020). Approximation by finite mixtures of
continuous density functions that vanish at infinity. \emph{Cogent
Mathematics \& Statistics}, \textbf{7}(1), Article 1750861.

\bibitem{NguyenEtAl2023} Nguyen, T., Chamroukhi, F., Nguyen, H. D.,
\& McLachlan, G. J. (2023). Approximation of probability density functions
via location-scale finite mixtures in Lebesgue spaces. \emph{Communications
in Statistics---Theory and Methods}, \textbf{52}(14), 5048--5059.

\bibitem{Promislow2015} Promislow, S. D. (2015). \emph{Fundamentals
of actuarial mathematics} (3rd ed.). Hoboken, NJ: Wiley.

\bibitem{Szasz1950} Szász, O. (1950). Generalization of S. Bernstein's
polynomials to the infinite interval. \emph{Journal of Research of
the National Bureau of Standards}, \textbf{45}(3), 239--245. \url{https://doi.org/10.6028/jres.045.024}

\bibitem{Tijms2003} Tijms, H. C. (2003). \emph{A first course in
stochastic models}. Chichester: Wiley.

\bibitem{Totik1983} Totik, V. (1983). Approximation by Szász--Mirakjan--Kantorovich
operators in $L^{p}$ $(p>1)$. \emph{Analysis Mathematica}, \textbf{9}(2),
147--167.

\bibitem{Totik1985} Totik, V. (1985). Saturation of Kantorovich type
operators. \emph{Periodica Mathematica Hungarica}, \textbf{16}(2),
115--126.

\bibitem{vanDerVaart1998} van der Vaart, A. W. (1998). \emph{Asymptotic
statistics}. Cambridge: Cambridge University Press.

\bibitem{VerbelenAntonioClaeskens2016} Verbelen, R., Antonio, K.,
\& Claeskens, G. (2016). Multivariate mixtures of Erlangs for density
estimation under censoring. \emph{Lifetime Data Analysis}, \textbf{22}(3),
429--455.

\bibitem{WillmotLin2011} Willmot, G. E., \& Lin, X. S. (2011). Risk
modelling with the mixed Erlang distribution. \emph{Applied Stochastic
Models in Business and Industry}, \textbf{27}(1), 2--16.

\bibitem{WillmotWoo2015} Willmot, G. E., \& Woo, J.-K. (2015). On
some properties of a class of multivariate Erlang mixtures with insurance
applications. \emph{ASTIN Bulletin: The Journal of the IAA}, \textbf{45}(1),
151--173.

\bibitem{WiperInsuaRuggeri2001} Wiper, M., Ríos Insua, D., \& Ruggeri,
F. (2001). Mixtures of gamma distributions with applications. \emph{Journal
of Computational and Graphical Statistics}, \textbf{10}(3), 440--454.

\bibitem{YadavMishraMeherMursaleen2022} Yadav, R., Mishra, V. N.,
Meher, R., \& Mursaleen, M. (2022). Statistical convergence of Szász--Mirakjan--Kantorovich-type
operators and their bivariate extension. \emph{Filomat}, \textbf{36}(17),
5895--5912.

\bibitem{YinEtAl2019} Yin, C., Lin, X. S., Huang, R., \& Yuan, H.
(2019). On the consistency of penalized MLEs for Erlang mixtures.
\emph{Statistics \& Probability Letters}, \textbf{145}, 12--20.

\bibitem{YoungEtAl2019} Young, D. S., Chen, X., Hewage, D. C., \&
Nilo-Poyanco, R. (2019). Finite mixture-of-gamma distributions: Estimation,
inference, and model-based clustering. \emph{Advances in Data Analysis
and Classification}, \textbf{13}(4), 1053--1082.

\end{thebibliography}
\end{document}